\documentclass[a4paper,11pt,reqno]{amsart}

\usepackage[latin1]{inputenc}
\usepackage[T1]{fontenc}
\usepackage[english]{babel}
\usepackage{latexsym,amsmath,amssymb,amsfonts}
\usepackage{dsfont}
\usepackage{color}
\usepackage{graphicx}
\usepackage{float}
\usepackage{esint}
\usepackage{multirow}

\theoremstyle{plain}
\begingroup
\newtheorem{theorem}{Theorem}[section]
\newtheorem{lemma}[theorem]{Lemma}
\newtheorem{proposition}[theorem]{Proposition}
\newtheorem{corollary}[theorem]{Corollary}
\endgroup
\theoremstyle{definition}
\begingroup
\newtheorem{definition}[theorem]{Definition}
\newtheorem{remark}[theorem]{Remark}

\endgroup
\theoremstyle{remark}

\newcommand{\R}{\mathbb R}

\newcommand{\dd}{\mathrm{d}}

\newcommand{\dx}{\;\dd x}
\newcommand{\dy}{\;\dd y}

\newcommand{\meno}{\!\setminus\!}
\newcommand{\e}{\varepsilon}
\newcommand{\p}{\varphi}
\renewcommand{\o}{\Omega}

\newcommand{\f}{\mathcal{F}}
\newcommand{\g}{\mathcal{G}}

\newcommand{\average}{{\mathchoice {\kern1ex\vcenter{\hrule height.4pt
width 6pt
depth0pt} \kern-9.7pt} {\kern1ex\vcenter{\hrule height.4pt width 4.3pt
depth0pt}
\kern-7pt} {} {} }}

\title[Piecewise constant images in dimension one]{Exact solutions for the denoising problem of piecewise constant images in dimension one}
\author{Riccardo Cristoferi}

\begin{document}

\begin{abstract}
In this paper we propose a method to determine explicitly the solution of the total variation denoising problem with an $L^p$ fidelity term, where $p>1$, for piecewise constant initial data in dimension one.
\end{abstract}

\maketitle


\section{Introduction}

When an image is acquired it comes, unavoidably, with some distortion. Indeed, external conditions, other then defects or limitations of the instruments that are used to obtain them, affect the quality of the acquired data. Thus, in order to be able to perform any task on the image, it is important to be able to recover the \emph{clean} version in the best possible way, \emph{i.e.}, with optimal fidelity.
If we denote it by $u$ and the acquired, corrupted image by $f$, it is usually assumed that
the two are related via:
\begin{equation}\label{eq:ip}
f=Au+n\,,
\end{equation}
where $A$ is a bounded linear operator representing the blurring effect and $n$ is the implementation of the random noise.
One of the aims of image reconstruction is deblurring and denoising $f$ in order to recover $u$ (see \cite{AubKor, ChaShe}).

Here we are interested in the denoising problem, \emph{i.e.}, when the operator $A$ is the identity and we have to remove the noise.
Problem \eqref{eq:ip} is, in general, ill-posed (in the sense of Hadamard) and thus we need to regularize it (see \cite{AcaVog, TikArs}).
A widely used variational technique for this purpose was introduced by Rudin, Osher and Fatemi in \cite{ROF}, where they proposed to recover $u$ in an open set $\o\subset\R^N$ via the minimization problem
\begin{equation}\label{eq:p}
\min_{u\in BV(\o),\,\,\|u-f\|^2_{L^2}=\sigma^2}\, |Du|(\o)\,,
\end{equation}
for some fixed $\sigma>0$, where $f$ is assumede to be in $L^2(\o)$ and $|Du|(\o)$ denotes the total variation of the function $u$ in $\o$. The choice of $BV(\o)$ as the functional space where to perform the minimization is motivated by the fact that it allows for the presence of discontinuities in the solutions representing the sharp edges of the objects in the image and the so called staircase effect due to the Cantor part of the derivative that takes case of the fine texture.
There are some interesting cases though, where the real image is represented by a function of bounded variation (see \cite{GouMor})). 
The minimization problem \eqref{eq:p} has been shown to be equivalent to the following \emph{penalized} minimization problem (known as the total variation denoising model with $L^2$ fidelity term)
\begin{equation}\label{eq:cl}
\min_{u\in BV(\o)} |Du|(\o) + \lambda\|u-f\|^2_{L^2(\o)}\,,
\end{equation}
for some Lagrange multiplier $\lambda>0$ (see \cite{ChaLio}).

Today's literature on the study of problem \eqref{eq:cl} is extensive, and here we limit ourselves to recall that
properties of the solutions have been studied, for instance, in \cite{All1, All2, All3, BelCasNov, BelCasNov2, CasChaNov, CasChaNov2, ChaEse, ChoFonZwi, DobSan, DuvAujGou, HadMey, Mey, Rin, Val, Ves}, the analysis of variants of \eqref{eq:cl} that use the generalized total variation have been performed in \cite{BreKunPoc, BreKunVal, PapBre, PapVal, PapSch}, anisotropic models are undertaken in \cite{ChoGenObe, DroBer, EseOsh, LouZenOscXin}, while the effects of considering high-order models have been investigated in \cite{ChaMarMul, ChaKanShe, DMFLM, JanPapSch, PapSch}.
Finally, other variants of \eqref{eq:p} have been addressed in \cite{AmbMas, AmbMas2, OshSolVes}, and algorithmic considerations may be found in \cite{BauColMor, Cha, ChaPoc, Nik}.\\

In this paper we study the one dimensional case where $f$ is a piecewise constant function,
and we generalize the $L^2$ fidelity term to an $L^p$ fidelity term, with $p\in[1,\infty)$.
To be precise, we consider the minimization problem
\begin{equation}\label{eq:minpbg}
\min_{u\in BV(\o)}\g(u)\,,
\end{equation}
where $\o:=(a,b)\subset\R$ and
\[
\g(u):=|Du|(\o)+\lambda\|u-f\|^p_{L^p(\o)}\,,
\]
for a given initial piecewise constant data $f$.
Our aim is to provide a method for solving the minimization problem \eqref{eq:minpbg} in the case $p>1$.

We next explain the main idea behind the strategy we propose.
The rigid structure of the initial data forces the solution to be piecewise constant itself, with jump set contained in the one of $f$ (see Corollary \ref{cor:structure}). Moreover, a simple truncation argument shows that the solution takes values within the minimum and the maximum of $f$.
Hence, the minimization problem \eqref{eq:minpbg} with $f$ of the form
\[
f(x)=\sum_{i=1}^k f_i\,\chi_{(x_{i-1},x_i)}(x)\,,\quad\quad f_i\in\R\,,
\]
is equivalent to the following minimization problem
\begin{equation}\label{eq:problemintro}
\min_{v\in Q}G(v)\,,
\end{equation} 
where $Q:=[\min f, \max f]^k$ and $G:\R^k\rightarrow\R$ is the function defined as
\[
G(v):=\sum_{i=2}^k |v_i-v_{i-1}| + \lambda \sum_{i=1}^k L_i|f_i-v_i|^p\,,
\]
with $v=(v_1,\dots,v_k)$ and $L_i:=x_i-x_{i-1}$. The function $G$ is convex but it lacks differentiability on the hyperplanes where $\{v_{i-1}=v_i\}$.
Thus, in principle, one should minimize the function $G$ over several compact regions and then compare all the minimum values
in order to find the global minimizer.\\
Our method aims at overcoming this difficulty. We will be able, for each $\lambda$, to predict \emph{a priori} - that is without knowing explicitly $u^\lambda$ (the minimizer of $G$ corresponding to the parameter $\lambda$) - what the relative position of each $u^\lambda_i$ with respect to $u^\lambda_{i-1}$ and $f_i$ will be.
Knowing that, it is possible to look for the minimizer $u^\lambda$ only in a specific region of $\R^k$, where the absolute values present in the expression of $G$ can be written explicitly.
Hence, $u^\lambda$ can be found by solving the appropriate Euler-Lagrange equation.

We give a more detailed description of our method: the function $\lambda\mapsto u^\lambda$ is continuous
and $u^\lambda\rightarrow f$ as $\lambda\rightarrow\infty$ (see Lemma \ref{lem:prop1}).
Hence, for $\lambda\gg1$, we have that $u^\lambda_i$ is very close to $f_i$, and
this allows us to tell the relative position of $u^\lambda_i$ with respect to $u^\lambda_{i-1}$.
Moreover, thanks to the qualitative properties of the solutions we will prove in Lemma \ref{lem:prop2} and Proposition \ref{prop:mono}, we will also be able to tell the relative position of each $u_i$ with respect to $f_i$. 
These information allow us to write explicitly the absolute values present in the expression of $G$, as well as to write explicitly the Euler-Lagrange equation, whose solution will give us the minimizer $u^\lambda$.
With this reasoning, we find the mininimizers for $\lambda$ large (how large it has to be will be determined a posteriori).

The idea now is to let $\lambda$ decrease. Since $u^\lambda$ is constant for small values of $\lambda$ (see Lemma \ref{lem:const}), by continuity of $\lambda\mapsto u^\lambda$ eventually two neighboring values $u^\lambda_i$ and $u^\lambda_{i-1}$ will happen to be the same. The main technical result (Theorem \ref{thm:split}) tells us that the same will be true for all smaller values of $\lambda$. As a result we now have to consider the function $G$ restricted to the subspace $\{v_{i-1}=v_i\}$, thus reducing the number of variables.
By continuity of $\lambda\mapsto u^\lambda$, it is then possible to predict the relative position of every $u_i$ with respect to $u_{i-1}$, while the qualitative properties of the solutions will give us the relative position of $u_i$ with respect to $f_i$.
As a consequence, also in this case, we are able to write explicitly the Euler-Lagrange equation.\\
We observe that price to pay for applying this method is that, in order to determine the solution of the minimization problem \eqref{eq:problemintro} for a certain value $\bar{\lambda}$, we first need to know it for all $\lambda>\bar{\lambda}$. This, in the end, boils down to solve some equations, whose number can be roughly bounded above by $k(k+1)/2$.\\

Our result is related to the work of Strong and Chan (see \cite{StrCha}), where the authors consider the minimization problem \eqref{eq:problemintro} in the special case $p=2$, but allowing the initial data $f$ to be a piecewise constant function with noise. Under certain conditions on the amplitude of the noise, they are able to determine the solution of the minimization problem \eqref{eq:problemintro} in the case $\lambda\gg1$.\\

Just a couple of words about the case $p=1$. The reason why the strategy described above fails for $p=1$ is because we cannot use the continuity of the map $\lambda\mapsto u^\lambda$. Indeed, even if for $p=1$ there is no uniqueness for the solution of the minimization problem \eqref{eq:problemintro} (see an example in Proposition \ref{prop:p1}), there is always a solution taking only the values that $f$ takes (see Corollary \ref{cor:p1}). But this jumping behavior of the solution prevents us to use continuity arguments, which are at the core of the strategy sketched above.
Nonetheless, the possibility of obtaining an analytic method for computing the solution in the case $p=1$ is currently under investigation.\\

The paper is organized as follows. After a brief recalling of the main properties of one dimensional functions of bounded variation in Section 2, we devote Section 3 to stating and proving basic results we will need in the sequel concerning the solutions of our minimization problem.
In Section 4 we illustrate with a simple case the different behaviors of the solution in the cases $p=1$ and $p>1$.
Section 5 contains the main technical results needed to justify the strategy to determine the solution of the minimization problem \eqref{eq:problemintro} we describe in Section 6 we conclude with an explicit example.


\section{Preliminaries}

In this section we review basic definitions of one dimensional functions of bounded variation. For more details, see \cite{AFP, Leo}. Here $a,b\in\R$ with $a<b$.

\begin{definition}
Let $u:(a,b)\rightarrow\R$. The \emph{pointwise variation} of $u$ in $(a,b)$ is defined as
\[
pV(u;a,b):=\sup\left\{\, \sum_{i=1}^{n-1}|u(x_{i+1})-u(x_i)| \,:\, a<x_1<\dots<x_n<b \,\right\}\,.
\]
\end{definition}
 
\begin{definition}
For $u\in L^1\bigl((a,b)\bigr)$ its \emph{total variation} in $(a,b)$ is given by
\[
|Du|\bigl((a,b)\bigr):=\sup\left\{\, \int_a^b \p'u \,\dd x \,:\, \p\in C^\infty_0\bigl((a,b)\bigr)\,, |\p|\leq1  \,\right\}\,.
\]
If $|Du|\bigl((a,b)\bigr)<\infty$, we say that $u$ belongs to the space $BV\bigl((a,b)\bigr)$ of functions of bounded variation in $(a,b)$.

In this case, $Du$ is a finite Radon measure on $(a,b)$.
\end{definition}

\begin{definition}
Let $u\in BV\bigl((a,b)\bigr)$. We define the \emph{jump set} of $u$ as
\[
J_u:=\bigl\{\, x\in(a,b) \,:\, |Du|(\{x\})\neq 0  \,\bigr\}\,.
\]
\end{definition}

The relation among the total and the pointwise variation is given by the following result.

\begin{theorem}\label{thm:ev}
Let $u\in L^1\bigl((a,b)\bigr)$ and define the \emph{essential variation} of $u$ as
\begin{equation}\label{eq:ev}
eV(u;a,b):=\inf\{\, pV(v;a,b) \,:\, v=u\quad L^1-a.e. \text{ in } (a,b) \,\}\,.
\end{equation}
The infimum defining $eV(u;a,b)$ in \eqref{eq:ev} is achieved and it coincides with $|Du|\bigl((a,b)\bigr)$.
\end{theorem}

Theorem \eqref{thm:ev} allows us to single out some well behaving representative of a BV function.

\begin{definition}
Let $u\in BV\bigl((a,b)\bigr)$. Any $v$ with $v=u$ $L^1$-a.e. in $(a,b)$ such that $pV(v;a,b)=eV(u;a,b)=|Du|\bigl((a,b)\bigr)$ is called a \emph{good representative} of $u$.
\end{definition}


\section{The general structure of the solutions}

This section is devoted to stating and proving some basic results we need concerning the solution of the minimization problem \eqref{eq:minpbg}. Albeit some of these properties may be known, we present here the proofs for the reader's convenience.\\

We start by proving that a solution to the minimization problem \eqref{eq:problemintro} with a piecewise constant initial data $f$ needs to have the same structure as $f$, \emph{i.e.}, it has to be a piecewise constant function with its jump set contained in the jump set of $f$.
In higher dimension, the inclusion $J_u\subset J_f$ is well known (see \cite{CasChaNov} and \cite{Val}) in the case $p>1$, while it is not always true if $p=1$ (see \cite{ChaEse} and \cite{DuvAujGou}).
The following result has been proved, with a different argument, in \cite{BreKunVal}.

\begin{theorem}
Let $f\in L^1\bigl((a,b)\bigr)$ and let $u\in BV\bigl((a,b)\bigr)$ be a solution of \eqref{eq:minpbg}.
If $f$ is constant in $(c,d)\subset(a,b)$, then $u$ is constant in $(c,d)$.
\end{theorem}

\begin{proof}
Let $u\in BV\bigl((a,b)\bigr)$ and suppose it is a good representative
such that
\[
u(c)=\lim_{y\rightarrow c^-}u(y)\,,\quad\quad
   u(d)=\lim_{y\rightarrow d^+}u(y)\,.
\]
Define the function
\[
\widetilde{u}:=\left\{
\begin{array}{ll}
u & \text{in } (a,b)\meno (c,d)\,,\\
t & \text{in } (c,d)\,,
\end{array}
\right.
\]
where $t:=\fint_c^d u$.
We claim that
\[
\f(\widetilde{u})\leq\f(u)\,,
\]
where equality holds if and only if $u\equiv t$ in $(c,d)$.
We show that the above inequality holds separately for each term of the energy.
The fact that the fidelity term decreases is due to Jensen's inequality. Indeed, recalling that $f$ is constant on $(c,d)$,
say $f\equiv\bar{f}$ in $(c,d)$, we have that
\[
\Bigl|\, \fint_c^d u(y)\dy - \bar{f}   \,\Bigr|^p
   = \Bigl|\, \fint_c^d \bigl(u(y) - f\bigr) \dy  \,\Bigr|^p
   \leq\fint_c^d |u(y)-\bar{f}|^p\dy\,,
\]
and, integrating both sides on $(c,d)$, we obtain
\[
\int_c^d|t-\bar{f}|^p\dx\leq\int_c^d|u(x)-\bar{f}|^p\dx\,,
\]
where the equality case holds if and only if $u\equiv t$ in $(c,d)$.

We now consider the total variation term. We have that
\[
|D\widetilde{u}|([c,d])= |u(c)-t|+|u(d)-t|\,,
\]
Suppose, without loss of generality, that $u(c)\leq u(d)$.
We will consider three cases: $t\in[u(c),u(d)]$, $t\leq u(c)$ and $t\geq u(d)$.
In the first one, we simply notice that
\[
|D\widetilde{u}|([c,d]) = u(d)-u(c)\leq  |Du|([c,d])\,.
\]
If $t\leq u(c)$, then there exists $x\in[c,d)$ such that
$u(x)\leq t$. Thus,
\begin{align*}
|Du|([c,d])&\geq \bigl(u(c)-u(x)\bigr)+\bigl(u(d)-u(x)\bigr)\geq \bigl(u(c)-t\bigr)+\bigl(u(d)-t\bigr)\\
&=|D\widetilde{u}|([c,d])\,.
\end{align*}
The case $t\geq \max\{u(c), u(d)\}$ can be treated similarly.
This concludes the proof.
\end{proof}

The above result allows us to get the structure of minimizers of problem \eqref{eq:minpbg} in the case in which $f$ is a piecewise constant function.

\begin{corollary}\label{cor:structure}
Let $f$ be a piecewise constant function in $(a,b)$, \emph{i.e.},
\[
f(x)=\sum_{i=1}^k f_i\,\chi_{(x_{i-1},x_i)}(x)\,,\quad\quad f_i\in\R\,.
\]
Then any solution $u$ of the minimization problem \eqref{eq:minpbg} is of the form
\begin{equation}\label{eq:u}
u(x)=\sum_{i=1}^k u_i\,\chi_{(x_{i-1},x_i)}(x)\,,
\end{equation}
for some $(u_i)_{i=1}^k\subset\R\,$, not necessarily distinct from each other.

In particular, a function $u$ of the form \eqref{eq:u}
is a solution of \eqref{eq:minpbg} if and only
$\bar{u}:=(u_1,\dots,u_k)\in\R^k$ is a solution of the minimization problem
\begin{equation}\label{eq:ming}
\min_{v\in\R^k}G(v)\,,
\end{equation} 
where $G:\R^k\rightarrow\R$ is the function defined as
\begin{equation}\label{eq:g}
G(v):=\sum_{i=2}^k |v_i-v_{i-1}| + \lambda \sum_{i=1}^k L_i|f_i-v_i|^p\,,
\end{equation}
where $v=(v_1,\dots,v_k)$ and $L_i:=x_i-x_{i-1}$.
\end{corollary}

Thus, we now concentrate on the study of the minimization problem \eqref{eq:ming}.\\
The cases $p=1$ and $p>1$ turn out to be quite different. Heuristically, the difference lies in the fact that, in the first case, the two terms of the energy are \emph{of the same order} while, for $p>1$, the fidelity term is of higher order than the total variation one. This leads to very different behavior of the solutions in the two cases.

One of the peculiar features of the case $p=1$ is the lack of uniqueness (see Proposition \ref{prop:p1}).
However, it is possible to identify a solution with a particular structure.

\begin{corollary}\label{cor:p1}
For $p=1$, there exists a solution $u$ of the problem \eqref{eq:ming} such that $u_i\in\{f_1,\dots,f_k\}$ for every $i=1,\dots,k$.
\end{corollary}

\begin{proof}
For any given quadruple of functions
\[
s_1:\{2,\dots,k\}\rightarrow\{0,1\}\,,\quad\quad s_2:\{1,\dots,k\}\rightarrow\{0,1\}\,,
\]
\[
t_1:\{2,\dots,k\}\rightarrow\{0,1\}\,,\quad\quad t_2:\{1,\dots,k\}\rightarrow\{0,1\}\,,
\]
let us consider the set $\mathcal{A}^{t_1,t_2}_{s_1,s_2}\subset\R^k$ such that
\begin{align}\label{eq:s1}
G(u)&=\sum_{i=2}^k (-1)^{s_1(i)}t_1(i)(u_i-u_{i-1})+\lambda \sum_{i=1}^k (-1)^{s_2(i)}t_2(i)L_i(f_i-u_i) \nonumber \\
&=v^{s_1,s_2,t_1,t_2}_\lambda\cdot u + c^{s_1,s_2,t_1,t_2}_\lambda\,,
\end{align}
for all $u\in\mathcal{A}^{t_1,t_2}_{s_1,s_2}$, where $c^{s_1,s_2,t_1,t_2}_\lambda\in\R$ and $v^{s_1,s_2,t_1,t_2}_\lambda\in\R^k$.
The result then follows by noticing that $G$ restricted to any $\mathcal{A}^{t_1,t_2}_{s_1,s_2}\subset\R^k$ is 
always minimized by a vector $u\in\R^k$ with
\[
u_i=f_{\sigma(i)}\,,
\]
for some function $\sigma:\{1,\dots,k\}\rightarrow\{1,\dots,k\}$ and that
\[
\min_{\R^k}G=\min_{s_1,s_2,t_1,t_2}\,\min_{\mathcal{A}^{t_1,t_2}_{s_1,s_2}}G_{|_{\mathcal{A}^{t_1,t_2}_{s_1,s_2}}}\,.
\]
\end{proof}

\begin{definition}
We will denote by $u^\lambda$ a solution of the minimization problem \eqref{eq:ming} corresponding to the value $\lambda$.
This will be \emph{the} solution, if $p>1$, while, for $p=1$, it will be understood as \emph{a} solution whose structure is those given by the previous result.
\end{definition}

\begin{remark}
It is easy to see that $u_i\in[\min f, \max f]$ for every solution $u$.
\end{remark}

In the rest of this section we seek to understand the behavior of the solution $u^\lambda$ in the limiting cases for $\lambda$,
\emph{i.e.}, when $\lambda\ll1$ and when $\lambda\gg1$.
In the first case the predominant term of the energy is given by the total variation, thus we expect $u^\lambda$ to minimizes it.

\begin{lemma}\label{lem:const}
Fix $p\geq 1$, positive numbers $(L_i)_{i=1}^k$ and two constants $m<M$.
Then, there exists a constant $\bar{\lambda}>0$, depending only on $p$, $(L_i)_{i=1}^k$, $m$ and $M$, with the following property.
For any piecewise constant function $f$ such that $f\in[m,M]$ and any
$\lambda\in(0,\bar{\lambda}]$, we have that $u^\lambda$ is constant.

In particular, if $p>1$ then there exists $c\in\R$ such that
$u^\lambda_i\equiv c$ for all $\lambda\in(0,\bar{\lambda}]$ and all $i=1,\dots,k$.
\end{lemma}

\begin{proof}
We first treat the case $p>1$. Assume that $u^\lambda$ is not constant and let $i\in\{1,\dots,k\}$ be such that $u^\lambda_i=\min\{u_j^\lambda\,:\,j=1,\dots,k\}$. Let
\[
r:=\inf\{j\leq i \,:\, u_s=u_i \text{ for all }  j\leq s\leq i\}\,,
\]
\[
t:=\sup\{j\geq i \,:\, u_s=u_i \text{ for all }  i\leq s\leq j\}\,.
\]
By hypothesis, either $r>1$ or $t<k$.
Consider, for $\e>0$, the  vector $u^\e\in\R^k$ defined as
$u^\e_j:=u_j+\e$ for $j=r,\dots,t$ and
$u^\e_j:=u^\lambda_j$ for all the other $j$'s.
Then, recalling that $u_j\in[m,M]$ for all $j=1,\dots,k$, we have that
\begin{align}\label{eq:der}
\lim_{\e\rightarrow0^+}\frac{G(u^\e)-G(u^\lambda)}{\e}&=a+p\lambda(-1)^{s_i}L_i|u_i-f_i|^{p-1} \nonumber\\
&\leq a+p\lambda (M-m)^{p-1}\max_{i=1,\dots,k} L_i\,,
\end{align}
where $a\in\{-1,-2\}$ (in particular, $a=-1$ if $r=1$ or $t=k$ and $a=-2$ otherwise), and $s_i\in\{0,1\}$. Let
\[
\bar{\lambda}:=\frac{1}{p(M-m)^{p-1}\max_i L_i}\,.
\]
If $\lambda<\bar{\lambda}$, from \eqref{eq:der} we get that
$G(u^\e)<G(u^\lambda)$. This means that $u^\lambda$ has to be constant for $\lambda<\bar{\lambda}$.
Moreover, it is easy to see that the function $G$ restricted to the set $\{ (u_1,\dots,u_k)\in\R^k \,:\, u_1=\dots=u_k \}$ admits a unique minimizer, that is independent of $\lambda$.

We now have to prove that $u^{\bar{\lambda}}$ is constant. Assume that $u^\lambda_i\equiv c$ for for all $\lambda\in(0,\bar{\lambda})$ and all $i=1,\dots,k$. Let $\bar{c}\in\R^k$ be the vector given by $\bar{c}_i:=c$.
Then $G_\lambda(c)<G_\lambda(v)$ for all $v\in\R^k$ with $v\neq\bar{c}$ and all $\lambda\in(0,\bar{\lambda})$,
where the subscript $\lambda$ is to underline the dependence of $G$ on $\lambda$.
By letting $\lambda\nearrow\bar{\lambda}$, we get $G_{\bar{\lambda}}(c)<G_{\bar{\lambda}}(v)$ for all $v\in\R^k$ and thus $u^{\bar{\lambda}}=\bar{c}$.

Let us now treat the case $p=1$. Suppose that $u^\lambda$ is not constant. Recalling that $u^\lambda_i\in\{f_1,\dots,f_k\}$, we have that
\[
|D u^\lambda|(\o)\geq \min_i|f_i-f_{i-1}|\,.
\]
On the other hand, for any function $v$ such that $v\equiv c\in[\min f,\max f]$ in $(a,b)$, it holds that
\[
G(v)\leq\lambda k (\max_i L_i)(M-m)\,.
\]
Set
\[
\bar{\lambda}:=\frac{\min_i|f_i-f_{i-1}|}{k (\max_i L_i)(M-m)}\,.
\]
For $\lambda<\bar{\lambda}$ the above estimates show that $u^\lambda$ must be constant.

Finally, in order to prove that also $u^{\bar{\lambda}}$ is constant, we reason as follows: we know that
$u^\lambda=\bar{c}^\lambda$ for $\lambda\in(0,\bar{\lambda})$, for some $\bar{c}_\lambda=(c_\lambda,\dots,c_\lambda)\in\R^k$.
Take $\lambda_n\nearrow\bar{\lambda}$. Since $c_{\lambda_n}\in[\min f,\max f]$, up to a not relabelled subsequence we have that $c_{\lambda_n}\rightarrow c$. We conclude that
\[
G_{\bar{\lambda}}\bigl((c,\dots,c)\bigr)\leq G_{\bar{\lambda}}(v)
\]
for all $v\in\R^k$.
\end{proof}

We now consider the case $\lambda\gg1$. Since
\[
\lambda L_i |u^\lambda_i-f_i|^p\leq G(u^\lambda)\leq G(f)<\infty\,,
\]
we know that
\begin{equation}\label{eq:conv}
u^\lambda\rightarrow f\quad\quad \text{ as } \lambda\rightarrow\infty\,.
\end{equation}
The following results underline another important difference between the cases $p=1$ and $p>1$. Indeed, if $p=1$ the limit \eqref{eq:conv} is reached for $\lambda<\infty$, while if $p>1$ only asymptotically.

\begin{lemma}\label{lem:minmax}
Let $p>1$ and assume that $f$ is not constant. Then $u^\lambda\in(\min f,\max f)$ for all
$\lambda>0$. In particular, $f$ can never be a solution of the minimization problem \eqref{eq:ming}.
\end{lemma}

\begin{proof}
We first prove that $u^\lambda$ cannot achieve the value $\min f$.
Assume that $u^\lambda_i=\min f$ for some $i\in\{1,\dots,k\}$. Let $r\leq i\leq s$ be such that $u_j=u_i$ for all $j=r,\dots,s$.
Consider, for $\e>0$, the vector $u^\e\in\R^k$ given by $u^\e_j:=u^\lambda_j+\e$ for $j=r,\dots,s$ and
$u^\e_j:=u^\lambda_j$ for all other $j$'s.
Then
\[
\lim_{\e\rightarrow0^+}\frac{G(u^\e)-G(u)}{\e}=
     a-p\lambda \sum_{j=r}^s L_j(f_j-u^\lambda_i)^{p-1}<0\,,
\]
where $a\in\{-1,-2\}$.
This is in contradiction with the minimality of $u^\lambda$.

With a similar argument it is possible to show that $u$ cannot achieve $\max f$.
\end{proof}

\begin{lemma}\label{lem:p1f}
Let $p=1$. Then there exists $\bar{\lambda}>0$ such that for all
$\lambda\geq\bar{\lambda}$ the solution of the minimization problem \eqref{eq:ming} is unique and is given by $f$ itself.
\end{lemma}

\begin{proof}
Suppose that there exists a sequence $\lambda_j\rightarrow\infty$ for which $u^{\lambda_j}_i\neq f_i$
for all $j$'s (this is possible, since $k$ is finite).
By recalling that $u^{\lambda_j}_i\in\{f_1,\dots,f_k\}$, setting
\[
\bar{\lambda}:=\frac{G(f)}{\min_i L_i\,\min_i|f_i-f_{i-1}|}\,,
\]
we have, for $\lambda_j>\bar{\lambda}$, that
\[
G(u^{\lambda_j})\geq \lambda_j L_i|u_i^{\lambda_j}-f_i|>G(f)\,,
\]
contradicting the minimality of $u^{\lambda_j}$.
\end{proof}


\section{Explicit solutions in a simple case}

Here we study the case where $k=2$. This analysis, albeit its simplicity, is important to underline some features that distinguish the behavior of the solution of the minimization problem \eqref{eq:minpbg} in the cases $p=1$ and $p>1$.

\begin{proposition}\label{prop:p1}
Let $f_1<f_2$. Then the solutions $u^\lambda$ of the minimization problem \eqref{eq:ming} in the case $p=1$ are the following:
\begin{itemize}
\item if $L_1>L_2$, set $\lambda^1_T:=\frac{1}{L_2}$. Then
\[
\left\{
\begin{array}{lll}
u^\lambda_1=u^\lambda_2=f_1 & & \text{ for } \lambda<\lambda^1_T\,,\\
&\\
u^\lambda_1=f_1, u^\lambda_2\in[f_1,f_2] & & \text{ for } \lambda=\lambda^1_T\,,\\
&\\
u^\lambda_1=f_1, u^\lambda_2=f_2 & & \text{ for } \lambda>\lambda^1_T\,,\\
\end{array}
\right.
\]
\item if $L_1=L_2$, set $\lambda^1_T:=\frac{1}{L_1}$. Then
\[
\left\{
\begin{array}{lll}
u^\lambda_1\in[f_1,f_2], u^\lambda_2\geq u_1 & & \text{ for } \lambda\leq\lambda^1_T\,,\\
&\\
u^\lambda_1=f_1, u^\lambda_2=f_2 & & \text{ for } \lambda>\lambda^1_T\,,\\
\end{array}
\right.
\]
\item if $L_1<L_2$, set $\lambda^1_T:=\frac{1}{L_1}$. Then
\[
\left\{
\begin{array}{lll}
u^\lambda_1=u^\lambda_2=f_2 & & \text{ for } \lambda<\lambda^1_T\,,\\
&\\
u^\lambda_1\in[f_1,f_2], u^\lambda_2=f_2 & & \text{ for } \lambda=\lambda^1_T\,,\\
&\\
u^\lambda_1=f_1, u^\lambda_2=f_2 & & \text{ for } \lambda>\lambda^1_T\,,\\
\end{array}
\right.
\]
\end{itemize}
\end{proposition}

\begin{proof}
It is easy to see that we must have $f_1\leq u_1\leq u_2\leq f_2$.
Thus, we consider the region
\begin{equation}\label{eq:t}
\mathcal{T}:=\{\, (u_1, u_2)\in\R^2 \;:\; f_1\leq u_1\leq u_2\leq f_2 \,\}\,,
\end{equation}
and we rewrite the function $G$ in $\mathcal{T}$ as
\[
G(\bar{u})=[\lambda L_1-1]u_1 + [1-\lambda L_2]u_2+\lambda[f_2 L_2-f_1 L_1] =v_\lambda\cdot u+c_\lambda\,.
\]
When minimizing $G$ in $\mathcal{T}$, we can drop the term $c_\lambda$.
Then, the minimizers, according to the position of the vector $\frac{v_\lambda}{|v_\lambda|}$ (well defined for all $\lambda$'s, except in the case $L_1=L_2$ and $\lambda=\frac{1}{L_1}$), are the following:

\begin{figure}[H]
\includegraphics[scale=0.8]{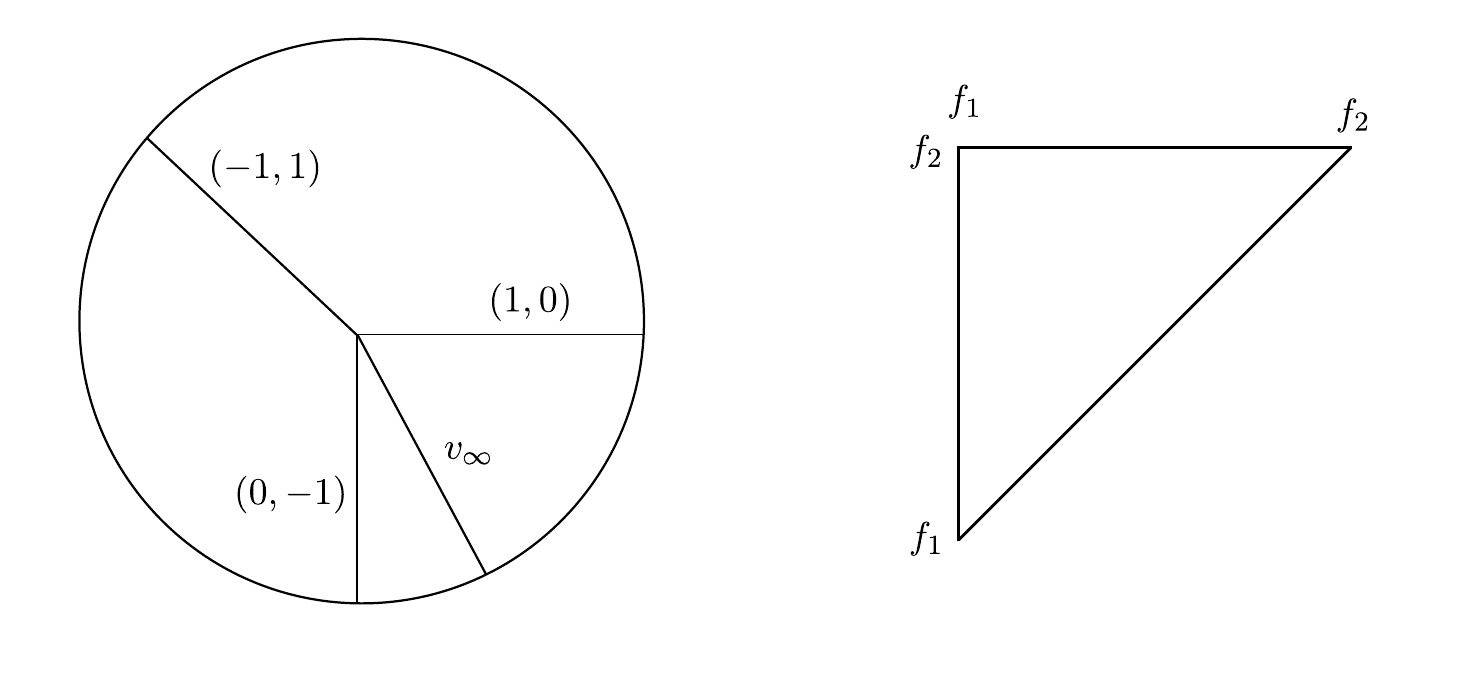}
\caption{On the left it is displayed where the (renormalized) vector $v_\lambda$ can vary: from $v_1$ for $\lambda=0$ up tp (asymptotically) $v_\infty:=\arctan\frac{L_2}{L_1}$. On the left the triangle where the vector $u$ can vary.}
\end{figure}

Thus, by simply studying the sign of the components of $v_\lambda$, we obtain the desired result.
Notice that the non uniqueness happens only when the vector $v_\lambda$ is orthogonal to $\{x=y\}\subset\R^2$.
\end{proof}

In the case $p>1$ the landscape of the solutions is quite different.

\begin{proposition}
Let $f_1<f_2$ and let $p>1$. Define
\[
\lambda^p_T:=\frac{1}{p}\frac{(L_1^{\frac{1}{p-1}}+L_2^{\frac{1}{p-1}})^{p-1}}{L_1 L_2(f_2-f_1)^{p-1}}\,.
\]
The solution $u^\lambda$ of the minimization problem \eqref{eq:ming} is the following:
\begin{itemize}
\item for $\lambda\leq \lambda_T^p$
\begin{equation}\label{eq:solp1}
u^\lambda_1=u^\lambda_2=\frac{L_1^{\frac{1}{p-1}}}{L_1^{\frac{1}{p-1}}+L_2^{\frac{1}{p-1}}}\,f_1 + \frac{L_2^{\frac{1}{p-1}}}{L_1^{\frac{1}{p-1}}+L_2^{\frac{1}{p-1}}}\,f_2\,,
\end{equation}
\item for $\lambda>\lambda_T^p$
\begin{equation}\label{eq:solp2}
u^\lambda_1= f_1+\frac{1}{(p\lambda L_1)^{\frac{1}{p-1}}}\,,\quad\quad
u^\lambda_2= f_2-\frac{1}{(p\lambda L_2)^{\frac{1}{p-1}}}\,.
\end{equation}
\end{itemize}
\end{proposition}

\begin{proof}
Recalling that $f_1\leq u_1\leq u_2\leq f_2$, we just have to consider the region $\mathcal{T}$ defined in \eqref{eq:t} and to rewrite the function $G$ in that region as
\[
G(u_1,u_2):=u_2-u_1+\lambda L_1(u_1-f_1)^p+\lambda L_2 (f_2-u_2)^p\,.
\]
The critical point of $G$ is given by
\[
u_1= f_1+\frac{1}{(p\lambda L_1)^{\frac{1}{p-1}}}\,,\quad\quad
u_2= f_2-\frac{1}{(p\lambda L_2)^{\frac{1}{p-1}}}\,,
\]
and it belongs to the interior of $\mathcal{T}$, \emph{i.e.}, $u_1^\lambda< u_2^\lambda$, only for
$\lambda>\lambda^p_T$. Since $G$ is strictly convex, this critical value turns out to be the global minimizer of $G$ for $\lambda>\lambda_T^p$. In the case $\lambda\leq\lambda_T^p$, the point of minimum has to be on $\partial\mathcal{T}$.
Instead of performing all the computations for finding the minimum point in all of the three edges of $\partial\mathcal{T}$ and to compare them, we will use the following argument based on the continuity of the minimizer $u^\lambda$ with respect to $\lambda$ (see Lemma \ref{lem:prop1}), \emph{i.e.}, we invoke the fact that the function $\lambda\mapsto u^\lambda$ is continuous.
Notice that for $\lambda\searrow\lambda_T^p$ we have
\[
u_\lambda\rightarrow (\bar{u},\bar{u})\,,
\]
where
\[
\bar{u}:=\frac{L_1^{\frac{1}{p-1}}}{L_1^{\frac{1}{p-1}}+L_2^{\frac{1}{p-1}}}\,f_1 + \frac{L_2^{\frac{1}{p-1}}}{L_1^{\frac{1}{p-1}}+L_2^{\frac{1}{p-1}}}\,f_2\,,
\]
is independent of $\lambda$. By using the continuity of the solution, we can conclude that,
for $\lambda\leq\lambda_T^p$, the solution of the minimization problem is given by $(\bar{u},\bar{u})$.
\end{proof}

\begin{remark}
We remark a couple of facts:
\begin{enumerate}
\item we have that $\lambda_T^p\rightarrow\lambda_T^1$ as $p\rightarrow1^+$ (in each of the cases for the definition of the second one). Indeed, suppose that $L_1<L_2$. Then,
\begin{align*}
\lim_{p\rightarrow1^+}\lambda_T^p &=
	\lim_{p\rightarrow1^+}\frac{(L_1^{\frac{1}{p-1}}+L_2^{\frac{1}{p-1}})^{p-1}}{L_1 L_2}\\
&= \frac{1}{L_1}\lim_{p\rightarrow1^+}\biggl(\, 1+\biggl(\frac{L_1}{L_2}\biggr)^{\frac{1}{p-1}} \,\biggr)^{p-1}\\
&= \frac{1}{L_1}\lim_{t\rightarrow0^+}exp\left[\, t\log\left[\left(\frac{L_1}{L_2}\right)^{\frac{1}{t}}  +1\right] \,\right]=\frac{1}{L_2}=\lambda_T^1\,.
\end{align*}
Similar reasonings lead to the claimed result in the other two cases.
In particular, notice that $\lambda_T^p>\lambda_T^1$.\\

\item The solutions that converge to \emph{a} solution for $p=1$, as $p\searrow1$.
Indeed, suppose $\lambda>\lambda_T^1$, Then for $p$ sufficiently close to $1$, from the above bullet point, we have that 
$\lambda>\lambda_T^p$. Thus, the solution of the minimization problem for $p$ is given by \eqref{eq:solp2}.
In this case, it is easy to see that the solution converges to $(f_1,f_2)$, as $p\searrow1$.
In the case $\lambda<\lambda_T^1$, we can assume as above that $p$ is so close to $1$ that the solution of the minimization problem for $p$ is given by \eqref{eq:solp1}.

If $L_1>L_2$, then
\[
\frac{L_1^{\frac{1}{p-1}}}{L_1^{\frac{1}{p-1}}+L_2^{\frac{1}{p-1}}}=\frac{1}{\Bigl( \frac{L_2}{L_1} \Bigr)^{\frac{1}{p-1}}+1}\rightarrow1\,,\quad\quad\text{ as } p\rightarrow1\,,
\]
\[
\frac{L_2^{\frac{1}{p-1}}}{L_1^{\frac{1}{p-1}}+L_2^{\frac{1}{p-1}}}=\frac{1}{\Bigl( \frac{L_1}{L_2} \Bigr)^{\frac{1}{p-1}}+1}\rightarrow0\,,\quad\quad\text{ as } p\rightarrow1\,.
\]
In the case $L_1=L_2$, both coefficients are equal to $\frac{1}{2}$.

Finally, in the case $\lambda=\lambda_T^1$, since $\lambda_T^p>\lambda_T^1$ we have that the solution of the minimization problem is given by \eqref{eq:solp1}. The result follows by arguing as before.
\end{enumerate}
\end{remark}

\begin{remark}
We expect a similar behavior for the minimization problem \eqref{eq:minpbg} in the case $p=1$ to hold also for general piecewise constant initial data $f$. In particular, we believe that the non uniqueness of the solution happens only for a finite number of critical values of $\lambda$, where a continuum of solutions is present. This set of critical values will be the set whose elements are $\lim_{p\rightarrow 1^+}\lambda^p_i$, where the $\lambda^p_i$ is the biggest value of $\lambda$ for which the solution $u$ corresponding to the parameters $\lambda$ and $p$ happens to have $u_i=u_{i+1}$ (see Theorem \ref{thm:split}).
\end{remark}


\section{The behavior of the solution for $p>1$}

This section contains the main result of this paper, namely Theorem \ref{thm:split}, that is derived from the qualitative properties of the solutions proved in the following two lemmas and in Proposition \ref{prop:mono}.

We start by proving the continuity of the solution $u^\lambda$ with respect to $\lambda$.

\begin{lemma}\label{lem:prop1}
Let $p>1$. Then
$\lambda\mapsto u^\lambda$ is continuous and $\lim_{\lambda\rightarrow\infty}u^\lambda=f$.
\end{lemma}

\begin{proof}
Fix $\bar{\lambda}>0$ and let $\lambda_n\rightarrow\lambda$.
Then $G(u^{\lambda_n})\leq G(v)$ for all $v\in\R^k$, where equality holds if and only if $v=u^{\lambda_n}$.
Since $|u^{\lambda_n}|\leq \sqrt{k}|\max_i f_i|$, up to a (not relabeled) subsequence, we have that
$u^{\lambda_n}\rightarrow\bar{v}$. Using the continuity of $G$ in both $v$ and $\lambda$,
we have that $G(\bar{v})\leq G(v)$ for all $v\in\R^k$. By the uniqueness of the solution, we deduce that
$\bar{v}=u^{\lambda}$, and that $u^{\lambda_n}\rightarrow u^\lambda$ for \emph{all} sequences $\lambda_n\rightarrow\lambda$.

To prove the second part of the lemma, we reason as follows. Assume that $u^\lambda$ does not converge to $f$ as $\lambda\rightarrow\infty$. Since $u_i\in[\min f,\max f]$, by compactness (up to a not relabelled subsequence) $u^\lambda\rightarrow v$, for some $v\neq f$. In particular, there exists an index $i$ such that $|u^\lambda_i-f_i|>\varepsilon$
for $\lambda\gg1$, for some $\varepsilon>0$. So that
\[
+\infty>G(f)\geq G(u^\lambda)\geq \lambda |u^\lambda_i-f_i|\rightarrow\infty\,,
\]
as $\lambda\rightarrow\infty$. This is the desired contradiction.
\end{proof}

We now prove several qualitative properties regarding the behavior of the solution $u^\lambda$ as $\lambda$ varies.
Some of the following results could be stated in a more inclusive way, but since they can be used to deduce qualitative properties of the solutions when no direct analysis can be performed, for clarity of exposition we opt to present each of them separately.

\begin{lemma}\label{lem:prop2}
Let $p>1$. Then, the following properties hold true:
\begin{itemize}
\item[(i)] Assume that, for $\lambda\in(\lambda_1,\lambda_2)$, there exists a function $\lambda\mapsto\bar{u}^\lambda$ such that, for some $r\geq0$,
\[
\left\{
\begin{array}{c}
u^\lambda_i=u^\lambda_{i+1}=\dots=u^\lambda_{i+r}=\bar{u}^\lambda\,,
 \\
 \\
u^\lambda_{i-1}<\bar{u}<u^\lambda_{i+r+1} \quad \text{ or } \quad u^\lambda_{i-1}>\bar{u}>u^\lambda_{i+r+1}\,.
\end{array}
\right.
\]

\begin{figure}[H]
\includegraphics[scale=1]{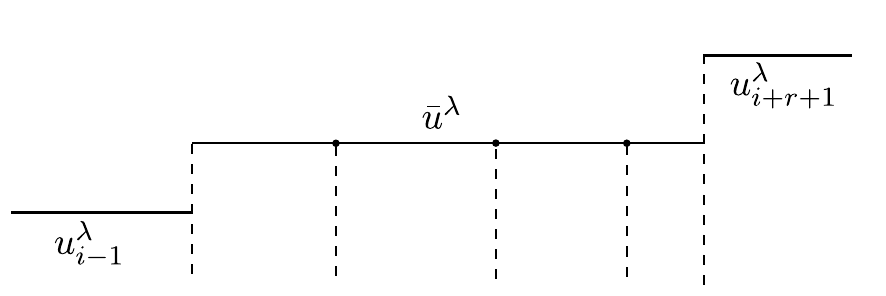}
\end{figure}

Then $\bar{u}^\lambda$ is the solution of
\[
\min_{c\in(u^\lambda_{i-1}, u^\lambda_{i+r+1})} \sum_{j=i}^{i+r} L_j|c-f_j|^p\,.
\]
In particular, $\bar{u}^\lambda$ is constant in $(\lambda_1,\lambda_2)$.\\

\item[(ii)] Assume that, for $\lambda\in(\lambda_1,\lambda_2)$, there exists a function $\lambda\mapsto\bar{u}^\lambda$ such that, for some $r\geq0$,
\[
\left\{
\begin{array}{c}
u^\lambda_i=u^\lambda_{i+1}=\dots=u^\lambda_{i+r}=\bar{u}^\lambda\,,\\
\\
u^\lambda_{i-1}\,,\, u^\lambda_{i+r+1}<\bar{u}^\lambda\,.
\end{array}
\right.
\]

\begin{figure}[H]
\includegraphics[scale=1]{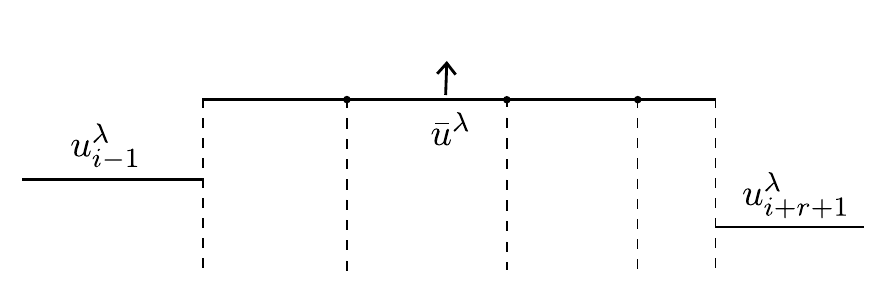}
\end{figure}

Then $\lambda\mapsto\bar{u}^\lambda$ is increasing.

In particular, in the case $r=0$, we have
\[
u^\lambda_i=f_i-\Bigl( \frac{2}{p\lambda L_i} \Bigr)^{\frac{1}{p-1}}\,.
\]
\medskip

\item[(iii)] Assume that, for $\lambda\in(\lambda_1,\lambda_2)$, there exists a function $\lambda\mapsto\bar{u}^\lambda$ such that, for some $r\geq0$,
\[
\left\{
\begin{array}{c}
u^\lambda_i=u^\lambda_{i+1}=\dots=u^\lambda_{i+r}=\bar{u}^\lambda\,,
\\
\\
u^\lambda_{i-1}\,,\, u^\lambda_{i+r+1}>\bar{u}^\lambda\,.
\end{array}
\right.
\]

\begin{figure}[H]
\includegraphics[scale=1]{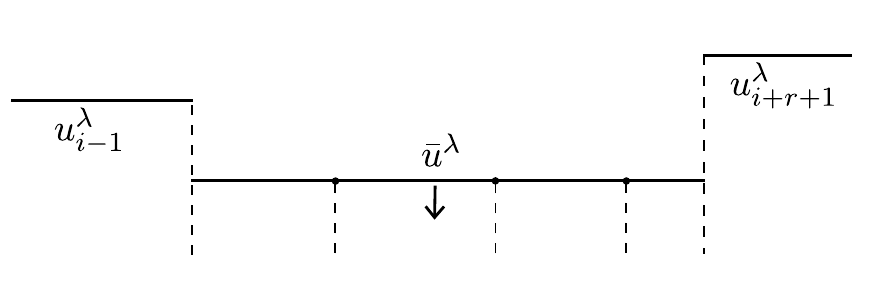}
\end{figure}

Then $\lambda\mapsto\bar{u}^\lambda$ is decreasing.

In particular, in the case $r=0$, we have
\[
u^\lambda_i=f_i+\Bigl( \frac{2}{p\lambda L_i} \Bigr)^{\frac{1}{p-1}}\,.
\]
\medskip

\item[(iv)] Assume that, for $\lambda\in(\lambda_1,\lambda_2)$, there exists a function $\lambda\mapsto\bar{u}^\lambda$ such that, for some $r\geq0$,
\[
\left\{
\begin{array}{c}
u^\lambda_1=u^\lambda_2=\dots=u^\lambda_{r}=\bar{u}^\lambda\,,\\
\\
u^\lambda_{r+1}<\bar{u}^\lambda\,.
\end{array}
\right.
\]

\begin{figure}[H]
\includegraphics[scale=1]{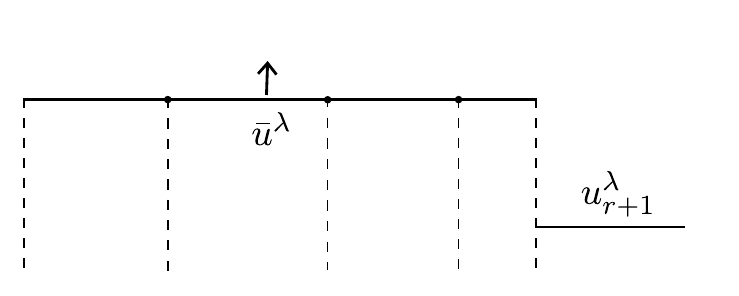}
\end{figure}

Then $\lambda\mapsto\bar{u}^\lambda$ is increasing.

In particular, in the case $r=0$, we have
\[
u^\lambda_i=f_1-\Bigl( \frac{1}{p\lambda L_1} \Bigr)^{\frac{1}{p-1}}\,.
\]
\medskip

\item[(v)] Assume that, for $\lambda\in(\lambda_1,\lambda_2)$, there exists a function $\lambda\mapsto\bar{u}^\lambda$ such that, for some $r\geq0$,
\[
\left\{
\begin{array}{c}
u^\lambda_1=u^\lambda_{i+1}=\dots=u^\lambda_{r}=\bar{u}^\lambda\,,\\
\\
u^\lambda_{r+1}>\bar{u}^\lambda\,.
\end{array}
\right.
\]

\begin{figure}[H]
\includegraphics[scale=1]{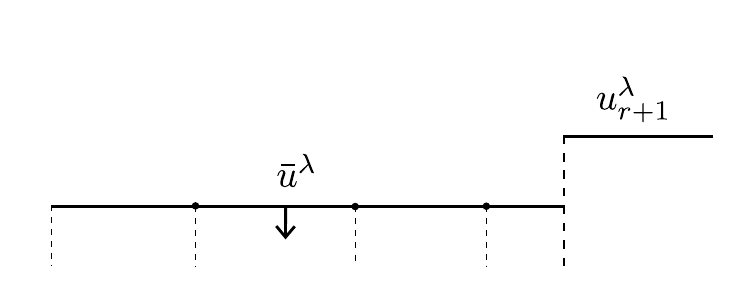}
\end{figure}

Then $\lambda\mapsto\bar{u}^\lambda$ is decreasing.

In particular, in the case $r=0$, we have
\[
u^\lambda_i=f_1+\Bigl( \frac{1}{p\lambda L_1} \Bigr)^{\frac{1}{p-1}}\,.
\]
\medskip

\item[(vi)] Assume that, for $\lambda\in(\lambda_1,\lambda_2)$, there exists a function $\lambda\mapsto\bar{u}^\lambda$ such that, for some $r\geq0$,
\[
\left\{
\begin{array}{c}
u^\lambda_{k-r}=\dots=u^\lambda_k=\bar{u}^\lambda\,,\\
\\
u^\lambda_{k-r-1}>\bar{u}^\lambda\,.
\end{array}
\right.
\]

\begin{figure}[H]
\includegraphics[scale=1]{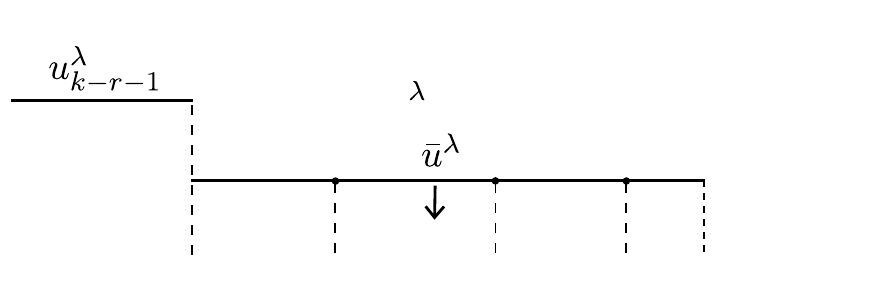}
\end{figure}

Then $\lambda\mapsto\bar{u}^\lambda$ is decreasing.

In particular, in the case $r=0$, we have
\[
u^\lambda_k=f_k+\Bigl( \frac{1}{p\lambda L_k} \Bigr)^{\frac{1}{p-1}}\,.
\]
\medskip

\item[(vii)] Assume that, for $\lambda\in(\lambda_1,\lambda_2)$, there exists a function $\lambda\mapsto\bar{u}^\lambda$ such that, for some $r\geq0$,
\[
\left\{
\begin{array}{c}
u^\lambda_{k-r}=\dots=u^\lambda_k=\bar{u}^\lambda\,,\\
\\
u^\lambda_{k-r-1}<\bar{u}^\lambda\,.
\end{array}
\right.
\]

\begin{figure}[H]
\includegraphics[scale=1]{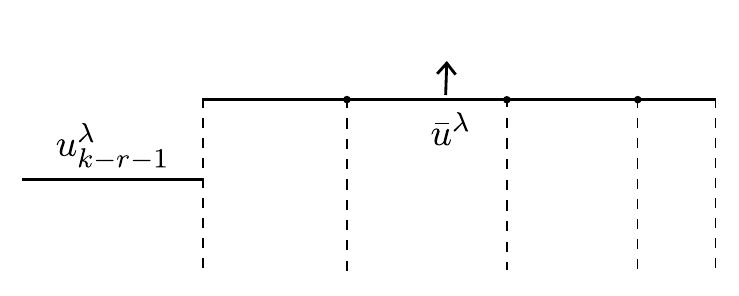}
\end{figure}

Then $\lambda\mapsto\bar{u}^\lambda$ is increasing.

In particular, in the case $r=0$, we have
\[
u^\lambda_k=f_k-\Bigl( \frac{1}{p\lambda L_k} \Bigr)^{\frac{1}{p-1}}\,.
\]
\end{itemize}
\end{lemma}

\begin{proof}
We start by proving property (i). Suppose that $u^\lambda_{i-1}<\bar{u}^\lambda<u^\lambda_{i+r+1}$. In the other case we argue in a similar way.
By hypothesis, the vector $u^\lambda$ minimizes the function $G$ in the set
\[
\{\, (u_1,\dots,u_k)\in\R^k : u_{i-1}<u_i=\dots=u_{i+r}<u_{i+r+1} \,\}\,,
\]
and in this set, the function $G$ can be written as
\[
G(u)=\widetilde{G}(u_1\dots,u_{i-1},u_{i+r+1},\dots,u_k)\\
+\lambda\sum_{j=i}^{i+r} L_j|\bar{u}-f_j|^p\,.
\]
By keeping $u_1,\dots,u_{i-1}$ and $u_{i+r+1},\dots,u_k$ fixed, the claim follows by minimizing the above quantity with respect to $\bar{u}$.\\

Since all the other properties can be proved with an argument whose general lines are similar, we just prove property (ii), leaving the details of the others proofs to the reader.

In the hypothesis of (ii), it holds that $u^\lambda$ is a minimizer of $G$ in the set
\[
\{\, (u_1,\dots,u_k)\in\R^k : u_{i-1},\, u_{i+r+1}<u_i=\dots=u_{i+r} \,\}\,.
\]
Restricted to this set, the function $G$ can be written as
\[
G(u)=\widetilde{G}(u_1\dots,u_{i-1},u_{i+r+1},\dots,u_k)\\
+2\bar{u}+\lambda\sum_{j=i}^{i+r} L_j|\bar{u}-f_j|^p\,.
\]
So, for $\lambda\in(\lambda_1,\lambda_2)$ and $u_1\dots,u_{i-1},u_{i+r+1},\dots,u_k$ fixed, $\bar{u}^\lambda$ is the minimizer of the strictly convex function 
\[
H(c):=2c+\lambda\sum_{j=i}^{i+r} L_j|c-f_j|^p
\]
in the set $(\max\{{u^\lambda_i}, u^\lambda_{i+r}\}, \max f)$.

To study the minimizer of $H$, we can assume without loss of generality that $f_i<f_{i+1}<\dots<f_{i+r}$.
Indeed, we notice that the order of the $f_j$'s doesn't matter. Moreover, in the case in which $f_p=f_q$ for some $p\neq q$, we can simply collect the two terms in a single one and use $L_p+L_q$ as a corresponding factor in the above summation.
We now want to prove that $\lambda\mapsto \bar{u}$ is decreasing.
Note that the function $H$ can be written as
\[
H(c)=2c+\lambda\sum_{j=i}^{m}L_j(c-f_j)^p + \lambda\sum_{j=m+1}^{i+r}L_j(f_j-c)^p=:H_m(c)\,,
\]
if $c\in(f_m,f_{m+1}]$, for some $m\in\{i,\dots,i+r-1\}$, and
\[
H(c)=2c+\lambda\sum_{j=i}^{i+r}L_j(c-f_j)^p\,,
\]
if $c\in[f_{i+r},\max f)$.
Consider the function $H_m$ in the interval $(f_m,f_{m+1})$. We have that
\[
H_m'(c)=2+p\lambda\left[\, \sum_{j=i}^{m}L_j(c-f_j)^{p-1} - \sum_{j=m+1}^{i+r}L_j(f_j-c)^{p-1} \,\right]\,.
\]
Here $H'_m(c)=0$ has a solution only if the term in the parenthesis is negative and if so, the let $\lambda\mapsto c^\lambda$ be such a solution. It is easy to see that this function is regular in
$(f_m,f_{m+1})$. By differentiating the expression $H'_m(c^\lambda)$ with respect to $\lambda$, we obtain
\begin{align*}
& p\left[\, \sum_{j=i}^{m}L_j(c-f_j)^{p-1} - \sum_{j=m+1}^{i+r}L_j(f_j-c)^{p-1} \,\right] \\
& + \lambda\frac{\dd c^\lambda}{\dd\lambda}p(p-1)\left[\, \sum_{j=i}^{m}L_j(c-f_j)^{p-2} + \sum_{j=m+1}^{i+r}L_j(f_j-c)^{p-2} \,\right]=0\,.
\end{align*}
Thus, by recalling that the term in the first parenthesis is negative, we get $\frac{\dd c^\lambda}{\dd\lambda}<0$, as desired.

In the case in which the minimizer of the function $H$ is reached at a point $c=f_{m+1}$, we simply consider the function
$H_m$ and we apply the argument above.

Finally, the same reasoning applies when $c\in[f_{i+r},\max f)$.
\end{proof}

We are now in position to prove the fundamental result we will use to develop our strategy for finding the solution.

\begin{theorem}\label{thm:split}
For each $i=1,\dots,k-1$ there exists $\lambda_i\in(0,\infty)$ such that $u_i^\lambda=u_{i+1}^\lambda$ for $\lambda\leq \lambda_i$, while $u_i^\lambda\neq u^\lambda_{i+1}$ for $\lambda>\lambda_i$.
\end{theorem}

\begin{proof}
\emph{Step 1.} We claim that if $u^{\widetilde{\lambda}}_i=u^{\widetilde{\lambda}}_{i+1}$ for some $\widetilde{\lambda}>0$, then $u^\lambda_i=u_{i+1}^\lambda$ for all $\lambda\in(0,\widetilde{\lambda}]$.
Indeed, let
\[
\bar{\lambda}:=\min\{\,\lambda \,:\, u^\mu_i=u^\mu_{i+1} \text{ fo all } \mu\in[\lambda,\widetilde{\lambda}]\,\}\,,
\]
and assume that $\bar{\lambda}>0$.
By continuity of $\lambda\mapsto u^\lambda$ there exists $\e>0$ such that $u_i^\lambda\neq u_{i+1}^\lambda$
for $\lambda\in(\bar{\lambda}-\e,\bar{\lambda})$.
Consider the case in which $u_i^\lambda<u_{i+1}^\lambda$ in $(\bar{\lambda}-\e,\bar{\lambda})$ (the other case can be treated similarly).

If $i=1$, then property (v) of Lemma \ref{lem:prop2} tells us that $\lambda\mapsto u_i^\lambda$ is decreasing in $(\bar{\lambda}-\e,\bar{\lambda})$ and thus it is not possible to have $u_i^{\bar{\lambda}}=u_{i+1}^{\bar{\lambda}}$.

If $i>1$, we can focus, without loss of generality, only on the following two cases: $u_{i-1}^\lambda>u_i^\lambda$ and
$u_{i-1}^\lambda<u_i^\lambda$ in $(\bar{\lambda}-\e,\bar{\lambda})$.

In the first case, we get a contradiction since by property (iii) of Lemma \ref{lem:prop2}, the map $\lambda\mapsto u^\lambda_i$ is decreasing in $(\bar{\lambda}-\e,\bar{\lambda})$ and thus, as above, we cannot have $u_i^{\bar{\lambda}}=u_{i+1}^{\bar{\lambda}}$.

In the other case, we have $u_{i-1}^\lambda<u_i^\lambda<u_{i+1}^\lambda$ in $(\bar{\lambda}-\e,\bar{\lambda})$. 
By using property (i) of Lemma \ref{lem:prop2}, we see that this is possible only if $u_i^\lambda=f_i$ for all
$\lambda\in(\bar{\lambda}-\e,\bar{\lambda})$. This yields the desired contradiction.\\

\emph{Step 2.} Let us define
\[
\lambda_i:=\max\{\, \lambda \,:\, u_i^\mu=u_{i+1}^\mu\,,\, \text{ for all } \mu\leq\lambda  \,\}\,.
\]
Step 1 and the continuity of $\lambda\mapsto u^\lambda$ ensure that $\lambda_i$ is well defined. Moreover, by Lemma \ref{lem:const}, we also get that $\lambda_i>0$ for all $i=i,\dots,k-1$. Finally, the fact that $u^\lambda\rightarrow f$ as $\lambda\rightarrow\infty$, tells us that $\lambda_i<\infty$
for all $i=1,\dots,k-1$. This concludes the proof.
\end{proof}

\begin{remark}
By a direct inspection of the proof of property (ii) of Proposition \ref{lem:prop2}, we see that the function
$\lambda\mapsto u^\lambda$ is continuous. Moreover, we can also say that it is smooth for all $\lambda\in(0,\infty)\meno S$,
where $S:=\{\lambda_1,\dots,\lambda_{k-1}\}\cup T$, where the $\lambda_i$'s are given by Theorem \ref{thm:split},
and $T:=\{\mu_1,\dots,\mu_k\}$ where $\mu_i:=\inf\{\lambda \,:\, u^\lambda_i=f_i \}$.
\end{remark}

Finally, we derive another consequence of Lemma \ref{lem:prop2} that will ensure that the solution is monotone where $f$ is and with the same monotonicity.

\begin{proposition}\label{prop:mono}
Suppose that $f_i<f_{i+1}<\dots<f_{i+r}$.
Then the solution $u$ of the minimization problem \eqref{eq:ming} is such that $u_i\leq u_{i+1}\leq\dots\leq u_{i+r}$.

In particular, $u$ has the following structure:
\begin{itemize}
\item if $u_i\geq f_{i+r}$, then $u_j=u_i$ for all $j=i,\dots,i+r$,
\item if $u_{i+r}\leq f_i$, then $u_j=u_{i+r}$ for all $j=i,\dots,i+r$,
\item otherwise, $u$ is of the form
\begin{equation*}
u_j=\left\{
\begin{array}{ll}
u_i & \text{ for } j=i,\dots,j_1\,, \\
f_j & \text{ for } j=j_1+1,\dots,j_2-1 \,, \\
u_{i+r} & \text{ for } j=j_2,\dots,k\,, \\
\end{array}
\right.
\end{equation*}
for some $f_{j_1}\leq u_i<f_{j_1+1}$ and
$f_{j_2}\leq u_{i+r}<f_{j_2+1}$, where $j_1<j_2$.
\end{itemize}

A similar statement holds in the case $f_i>f_{i+1}>\dots>f_{i+r}$. 
\end{proposition}

\begin{proof}
\emph{Step 1.} We claim that $u_i\leq u_{i+1}\leq\dots\leq u_{i+r}$.

Suppose that $u_{j-1}>u_j$ for some $j\in\{i+1,\dots,i+r\}$.
We have to treat three cases: $u_j<f_j$, $u_j=f_j$ and $u_j>f_j$.

In the first case, we get a contradiction with the minimality of $u^\lambda$ since it is easy to see that
\[
G(u^\lambda_1,\dots,u^\lambda_{j-1},u^\lambda_j+\e,u^\lambda_{j+1},\dots,u_k)<G(u^\lambda)\,,
\]
for $\e>0$ small.

Now, suppose $u_j>f_j$ and that $u_j>u_{j+1}$. Then, for $\e>0$ small,
\[
G(u^\lambda_1,\dots,u^\lambda_{j-1},u^\lambda_j-\e,u^\lambda_{j+1},\dots,u_k)<G(u^\lambda)\,,
\]
yielding the desired contradiction.

Finally, we can treat all the remaining cases (namely when $u_j=f_j$ or the case where $u_j>f_j$ and $u_{j+1}>u_j$) simultaneously as follows: let us denote by $j_m\in\{i,\dots,j\}$ be the minimum index $r$ such that $u_r>u_{r+1}$. In both cases we have $u_{j_m}>f_{j_m}$, and thus, 
\[
G(u^\lambda_1,\dots,u^\lambda_{j_m-1},u^\lambda_{j_m}-\e,u^\lambda_{j_m+1},\dots,u_k)<G(u^\lambda)\,,
\]
for $\e>0$ small.\\

\emph{Step 2.} Using Step 1, we have that
\[
\sum_{j=i+1}^{i+r} |u^\lambda_j-u^\lambda_{j-1}|=u^\lambda_{i+r}-u^\lambda_i\,.
\]
Since this value is invariant under modification of $u^\lambda_j$ for $j=i+1,\dots,i+r-1$,
if we keep $u_i$ and $u_{i+r}$ fixed, the minimality of $u^\lambda$ implies that
\[
\sum_{j=i}^{i+r}|u_j-f_j|^p=\min_\mathcal{A}\sum_{j=i}^{i+r}|v_i-f_i|^p\,,
\]
where
\[
\mathcal{A}:=\{ (v_{i+1},\dots,v_{i+r-1})\in\R^{i+r-2} \,:\, u_i\leq v_{i+1}\leq\dots\leq v_{i+r-1}\leq u_{i+r}\}\,.
\]
This proves the second part of the statement of the proposition.
\end{proof}


\section{A method for finding the solution.}

In this section we describe the method we propose in order to identify the solution of the minimization problem \eqref{eq:ming} in the case $p>1$.

The general idea is, for every $\lambda>0$, to be able to tell \emph{a priori} the relative position of each $u^\lambda_i$ with respect to $u^\lambda_{i-1}$ and $f_i$. Knowing that allows us to:
\begin{itemize}
\item[(i)]  know if the minimization of $G$ has to take place in some subspace $\{v_{i_1-1}=v_{i_1}\}\cap\dots\cap\{v_{i_r-1}=v_{i_r}\}$, and hence if we have to reduce the number of variables $G$ depends on,
\item[(ii)] write explicitly the absolute values present in the expression of $G$.
\end{itemize}
If we are able to do that, we can reduce the problem of minimizing the functional $G$ to the problem of minimizing a strictly convex functional of class $C^1$, and thus the minimizer can be found by solving the appropriate Euler-Lagrange equation.\\

Let us now explain how we are able to make our prediction.
For the sake of clarity, let us assume that our initial data $f$ is like in the figure below.

\begin{figure}[H]
\includegraphics[scale=0.6]{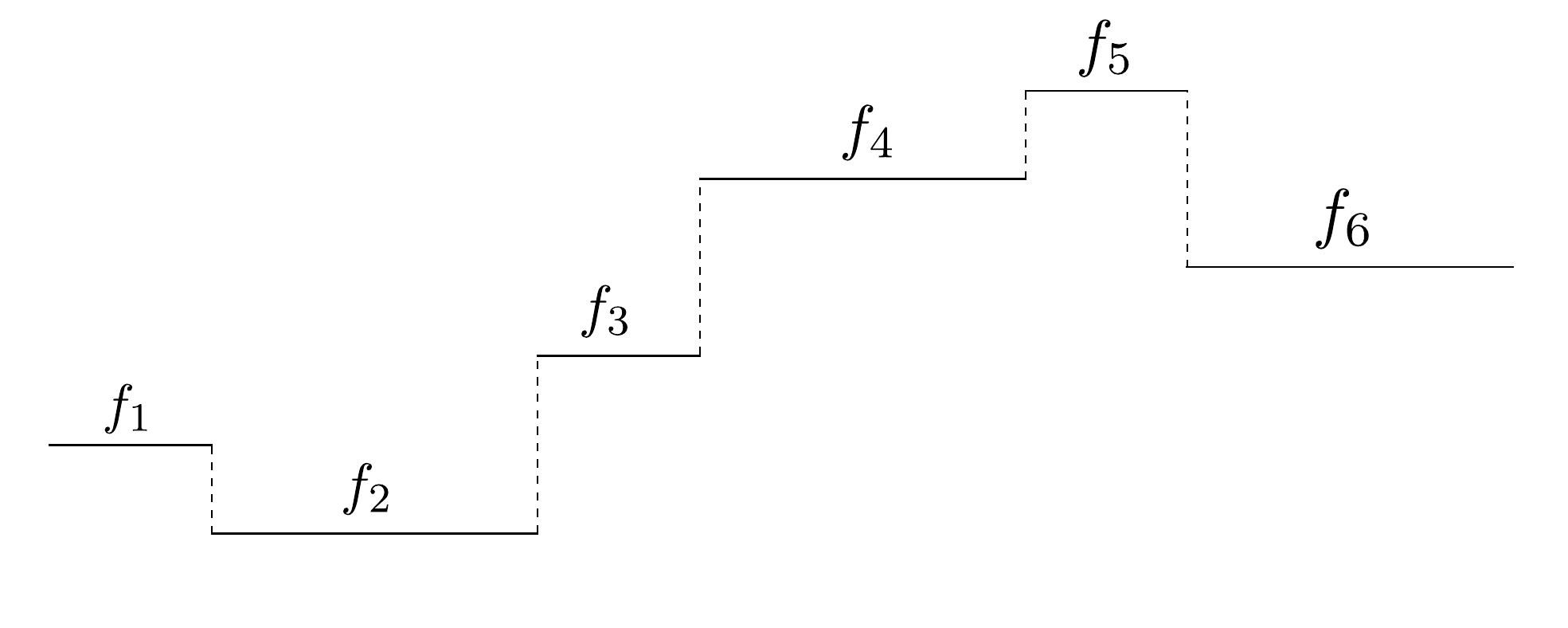}
\caption{The initial data $f$.}
\end{figure}

\emph{Step 1: Solutions for $\lambda$ large.}
By Proposition \ref{lem:prop1} we know that for $\lambda\gg1$ the solution $u^\lambda$ is such that $u_i$ is very close to $f_i$. In particular $u_{i-1}\neq u_i$ for every $i=2,\dots,k$.
Moreover, we also know what the relative position of $u^\lambda_i$ with respect to $u^\lambda_{i-1}$ is.
Using the properties given by Lemma \ref{lem:prop2} and Proposition \ref{prop:mono}, will also allow us to know the relative position of each $u^\lambda_i$ with respect to $f_i$.

\begin{figure}[H]
\includegraphics[scale=0.6]{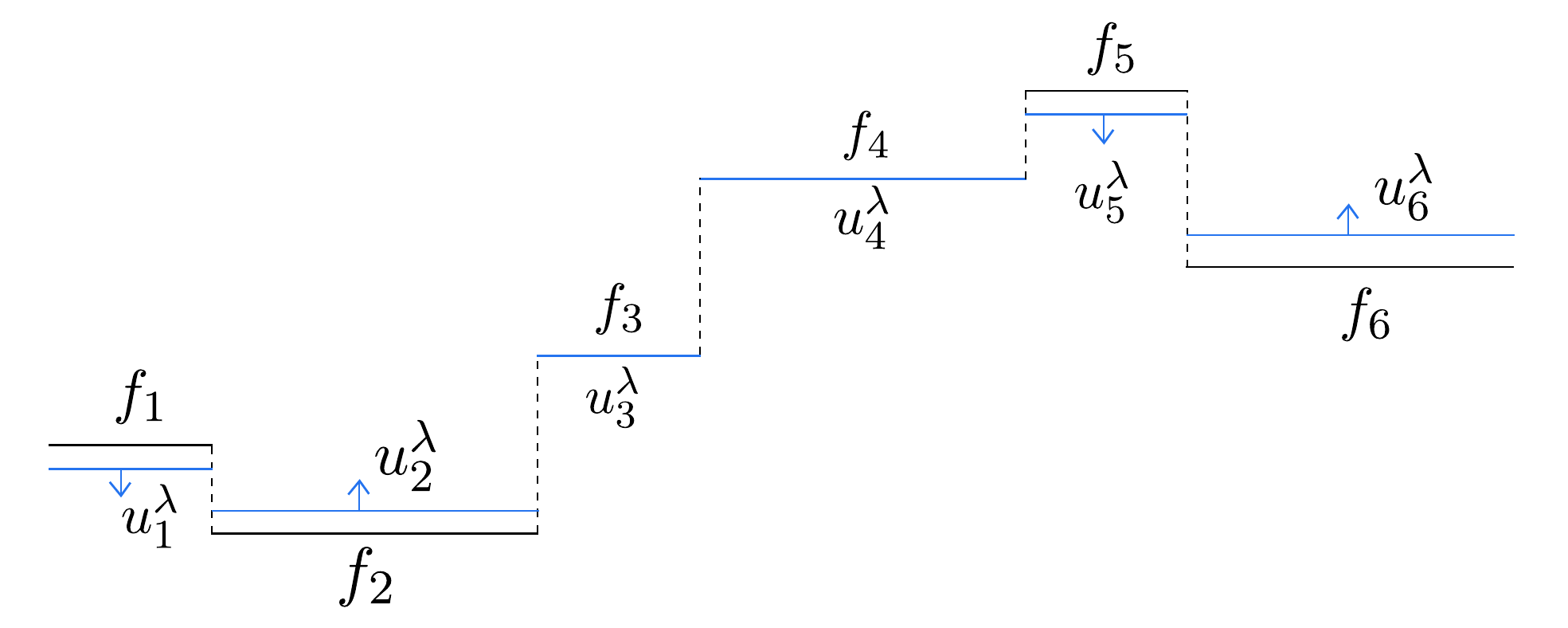}
\caption{The behavior of the solution for $\lambda\gg1$ as $\lambda$ decreases.}
\end{figure}

Thus, we are able to determine $\bar{r}_i,\bar{s}_i,\bar{t}_i\in\{0,1\}$ for which
\[
G(u^\lambda)=\sum_{i=2}^k (-1)^{\bar{s}_i}(u^\lambda_i-u^\lambda_{i-1})+\lambda\sum_{i=1}^k \bar{r}_i L_i \bigl((-1)^{\bar{t}_i}(f_i-u^\lambda_i)\bigr)^p\,,
\]
for all $\lambda\gg1$. This tells us that we can find $u^\lambda$ by solving the Euler-Lagrange equation of the function
\[
G(v)=\sum_{i=2}^k (-1)^{\bar{s}_i}(v_i-v_{i-1})+\lambda\sum_{i=1}^k L_i \bigl((-1)^{\bar{t}_i}(f_i-v_i)\bigr)^p\,.
\]
In particular, for all $i$ such that $\bar{r}_i\neq 0$ - for which we already know that $u^\lambda_i=f_i$ - we get
\[
u^\lambda_i=f_i-(-1)^{\bar{t}_i}\left[\, \frac{(-1)^{\bar{s}_i}-(-1)^{\bar{s}_{i-1}}}{p\lambda L_i}  \,\right]^{\frac{1}{p-1}}\,.
\]
Notice that how big $\lambda$ has to be in order to apply what we said above will be determined explicitly in the next step.\\

\emph{Step 2. Solutions for all smaller $\lambda$'s}. We now let $\lambda$ decrease. Since for small $\lambda$'s we know that the solution $u^\lambda$ is constant, the continuity of the function $\lambda\mapsto u^\lambda$ implies that, eventually, a \emph{critical} event will happen. That is, two neighboring values of $u^\lambda$ will coincide.
Notice that multiple critical events can happen simultaneously.

\begin{figure}[H]
\includegraphics[scale=0.6]{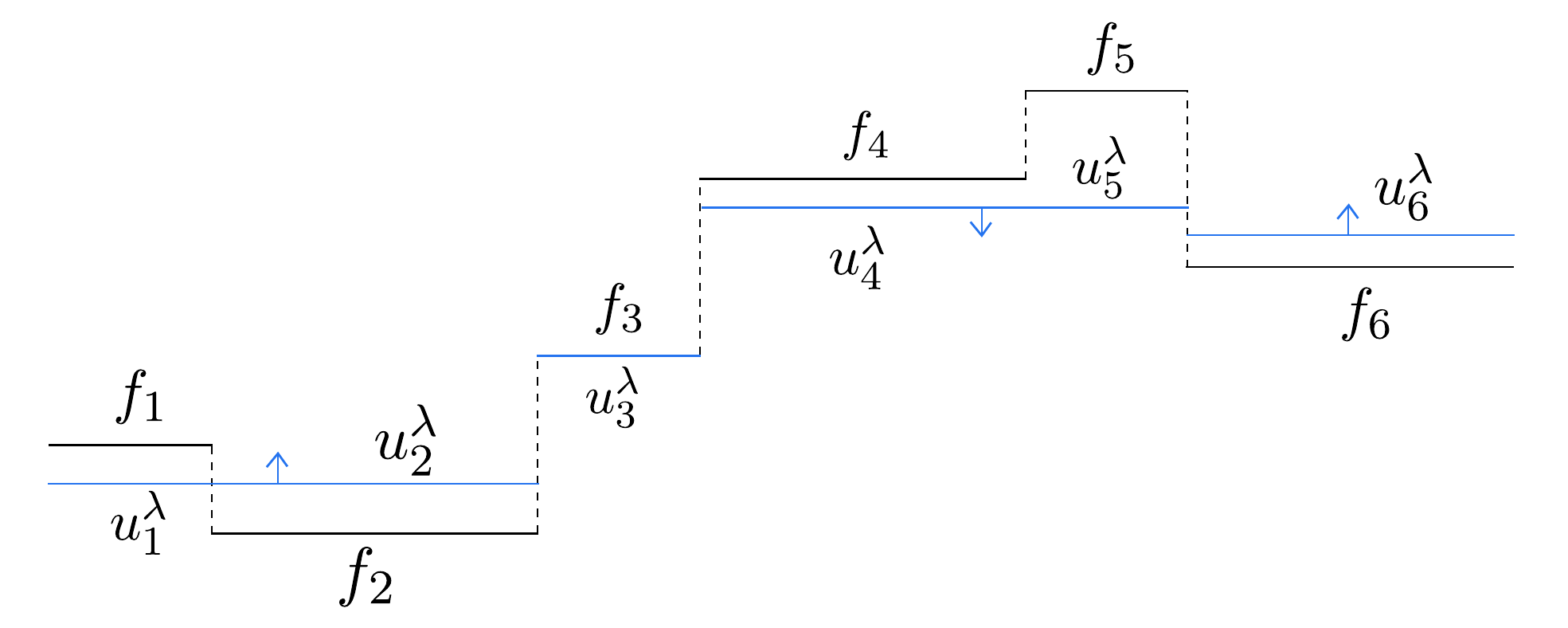}
\caption{The behavior of the solution for $\lambda$ after the first critical event as $\lambda$ decreases.}
\end{figure}

Assume, for instance, that for some $\lambda_j$, $u^{\lambda_j}_j$ happens to be equal to $u^{\lambda_j}_{j-1}$ and that this is the only critical event taking place.
By Theorem \ref{thm:split} we know that the same will be true for all smaller $\lambda$'s, that is $u^\lambda_j=u^\lambda_{j-1}$ for all  $\lambda\leq\lambda_j$. So that we now have to consider the functional $G$ restricted to the subspace $\{v_j=v_{j-1}\}$, that is
\[
\widetilde{G}(w_1,\dots,w_{k-1}):=G(w_1,\dots,w_{j-1},w_j,w_j,w_{j+1},\dots,w_{k-1})\,.
\]
In particular, we get a reduction of the number of variables $G$ depends on. Notice that, for $\lambda=\lambda_j$, the function $\widetilde{G}$ has no issues of differentiability. So that the minimizer $u^\lambda$ can be found by solving the Euler-Lagrange equation for $\widetilde{G}$.\\

\emph{Step 3. Find all the solutions.} We just repeat Step 2 until $u^\lambda_i=u^\lambda_i$ for all index $i=1,\dots,k-1$.
That would be the value of $\lambda$ for which we stop, Indeed, by Lemma \ref{lem:const}, we know that the solution will remain constant for all smaller values of $\lambda$.

Notice that, after the first critical event described in Step 2, another kind of critical event can take place. Namely, it can happen that some $u^\lambda_i$ will change its relative position with respect to $f_i$. If that happens, we just have to change the corresponding $\bar{t}_i$ and/or $\bar{r}_i$.\\

\medskip

\textbf{Example.} We illustrate the above strategy with an example.
For simplicity, we will treat the case $p=2$.

Suppose that $k=6$, take
\[
L_1=L_3=L_5=1\,,\quad\quad L_2=L_4=L_6=2\,.
\]
Consider the initial data $f$ given by
\[
f_1=2,\quad f_2=1,\quad f_3=3,\quad f_4=5,\quad f_5=6,\quad f_6=4\,.
\]

\begin{figure}[H]
\includegraphics[scale=0.6]{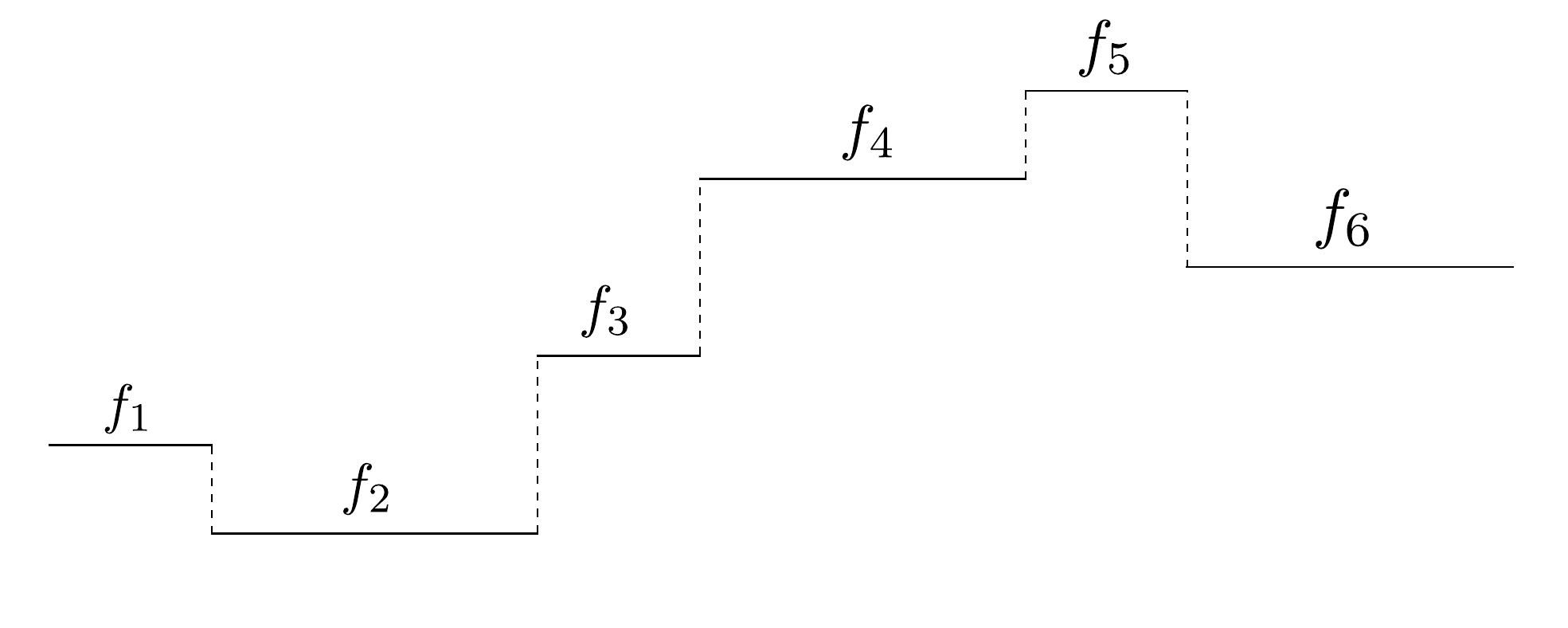}
\caption{The initial data $f$.}
\end{figure}

For $\lambda\gg1$, we know that we have to consider the following functional
\begin{align*}
G(u_1,u_2,u_3,u_4,u_5,u_6)&:= u_1-2u_2+2u_5-u_6+\lambda[\,(2-u_1)^2
    +2(1-u_2)^2\\
&\hspace{0.5cm}+|u_3-3|^2+2|u_4-5|^2+(6-u_5)^2+2(u_6-4)^2\,]\,.
\end{align*}
In particular, we obtain that the solution $u^\lambda$ is given by
\[
\begin{array}{lllll}
u_1^\lambda:=2-\frac{1}{2\lambda}\,,&&
u_2^\lambda:=1+\frac{1}{2\lambda}\,,&&
u_3^\lambda:=3\,,\\
&&&&\\
u_4^\lambda:=5\,,&&
u_5^\lambda:=6-\frac{1}{\lambda}\,,&&
u_6^\lambda:=4+\frac{1}{4\lambda}\,.
\end{array}
\]
for $\lambda>1$.

\begin{figure}[H]
\includegraphics[scale=0.6]{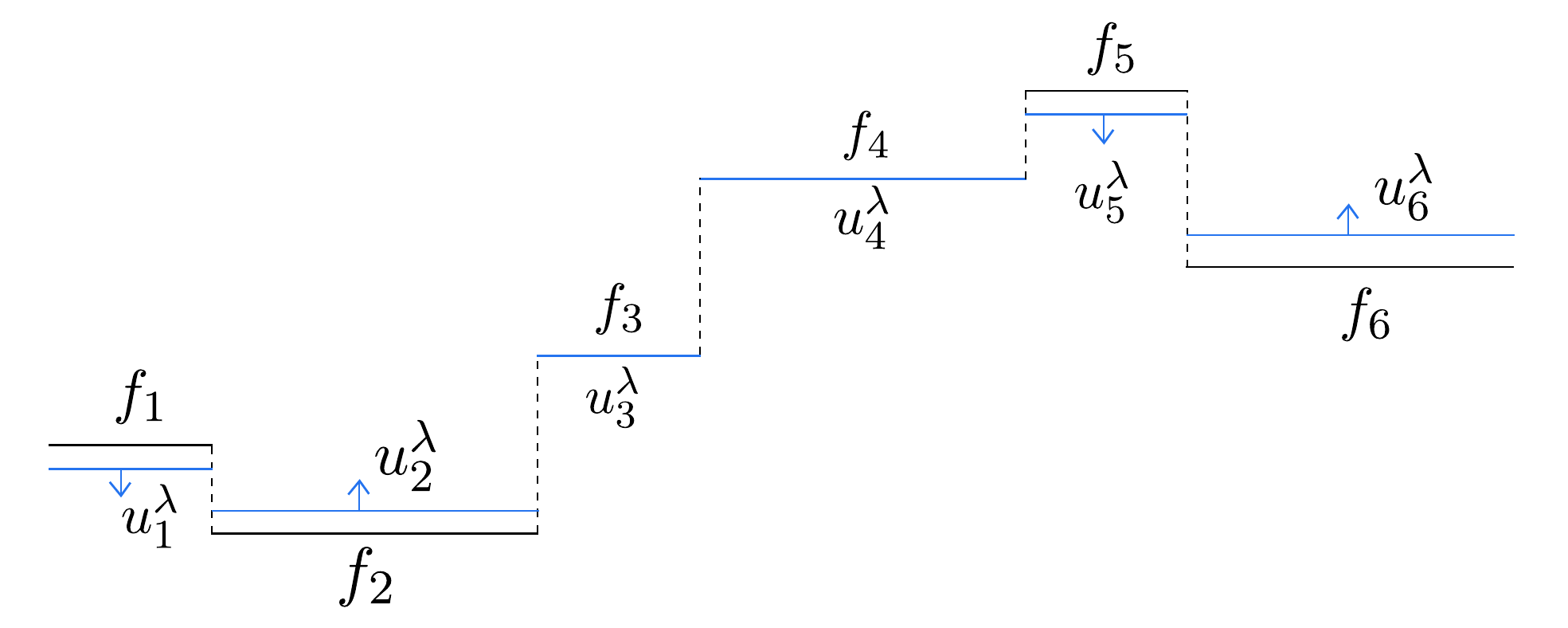}
\caption{The behavior of the solution for $\lambda>1$ as $\lambda$ decreases.}
\end{figure}

The first \emph{critical event} happens for $\lambda=1$, when
$u_1^\lambda=u_2^\lambda$ and $u_4^\lambda=u^\lambda_5$.
For smaller values of $\lambda$, we have to consider the functional
\begin{align*}
G(v_1,v_2,v_3,v_4)&:= 2v_3-v_1-v_4 + \lambda[\, (2-v_1)^2+
    2(v_1-1)^2+|v_3-3|^2\\
&\hspace{0.5cm}+2(5-v_3)^2+(6-v_3)^2+2(v_4-4)^2 \,]\,.
\end{align*}
Here, the solution is given by
\[
\begin{array}{lllll}
u_1^\lambda:=\frac{4}{3}+\frac{1}{6\lambda}\,,&&
u_2^\lambda:=\frac{4}{3}+\frac{1}{6\lambda}\,,&&
u_3^\lambda:=3\,,\\
&&&&\\
u_4^\lambda:=\frac{16}{3}-\frac{1}{3\lambda}\,,&&
u_5^\lambda:=\frac{16}{3}-\frac{1}{3\lambda}\,,&&
u_6^\lambda:=4+\frac{1}{4\lambda}\,.
\end{array}
\]

\begin{figure}[H]
\includegraphics[scale=0.6]{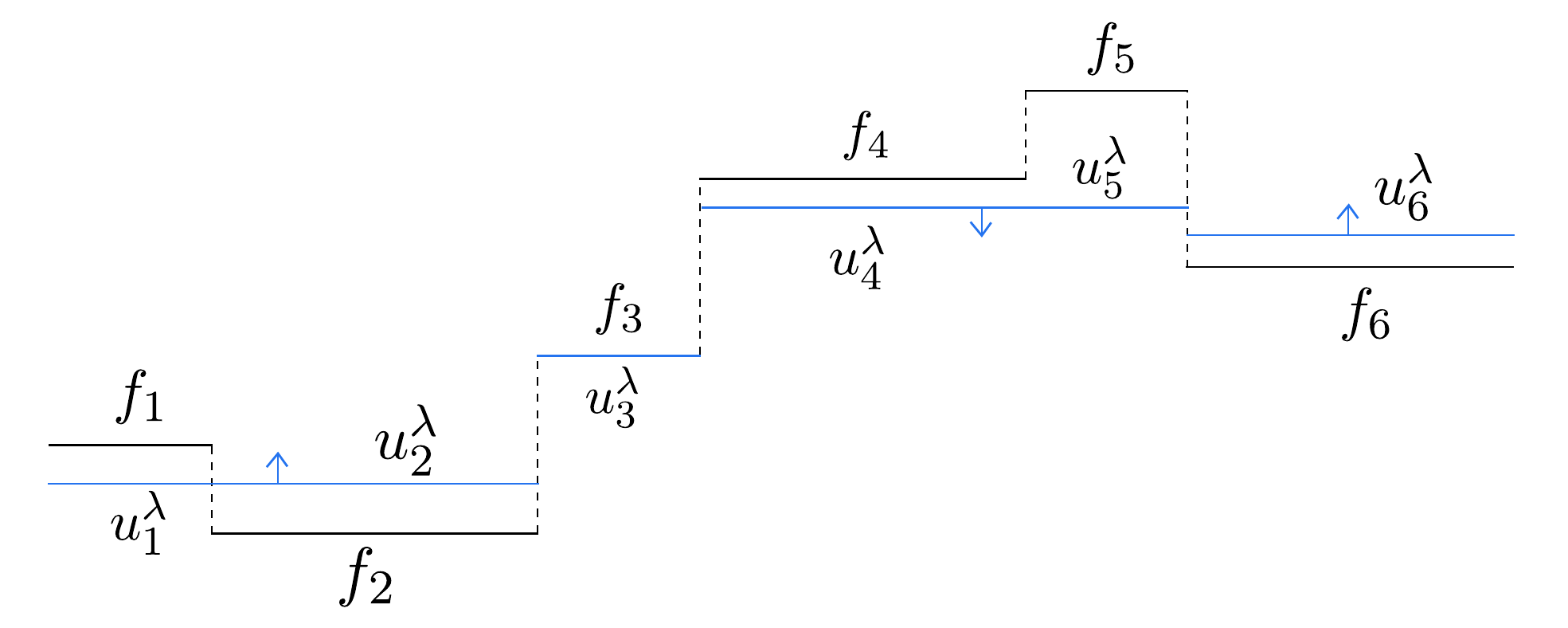}
\caption{The behavior of the solution for $\lambda\in(\frac{9}{14},1]$ as $\lambda$ decreases.}
\end{figure}

Then, for $\lambda=\frac{9}{14}$ we have that
$u_6^\lambda=u_5^\lambda$. Hence, the new functional we have to consider is
\begin{align*}
G(v_1,v_2,v_3)&:= v_3-v_1 + \lambda[\, (2-v_1)^2+
    2(v_1-1)^2+|v_2-3|^2\\
&\hspace{0.5cm}+2(5-v_3)^2+(6-v_3)^2+2(v_3-4)^2 \,]\,.
\end{align*}

\begin{figure}[H]
\includegraphics[scale=0.6]{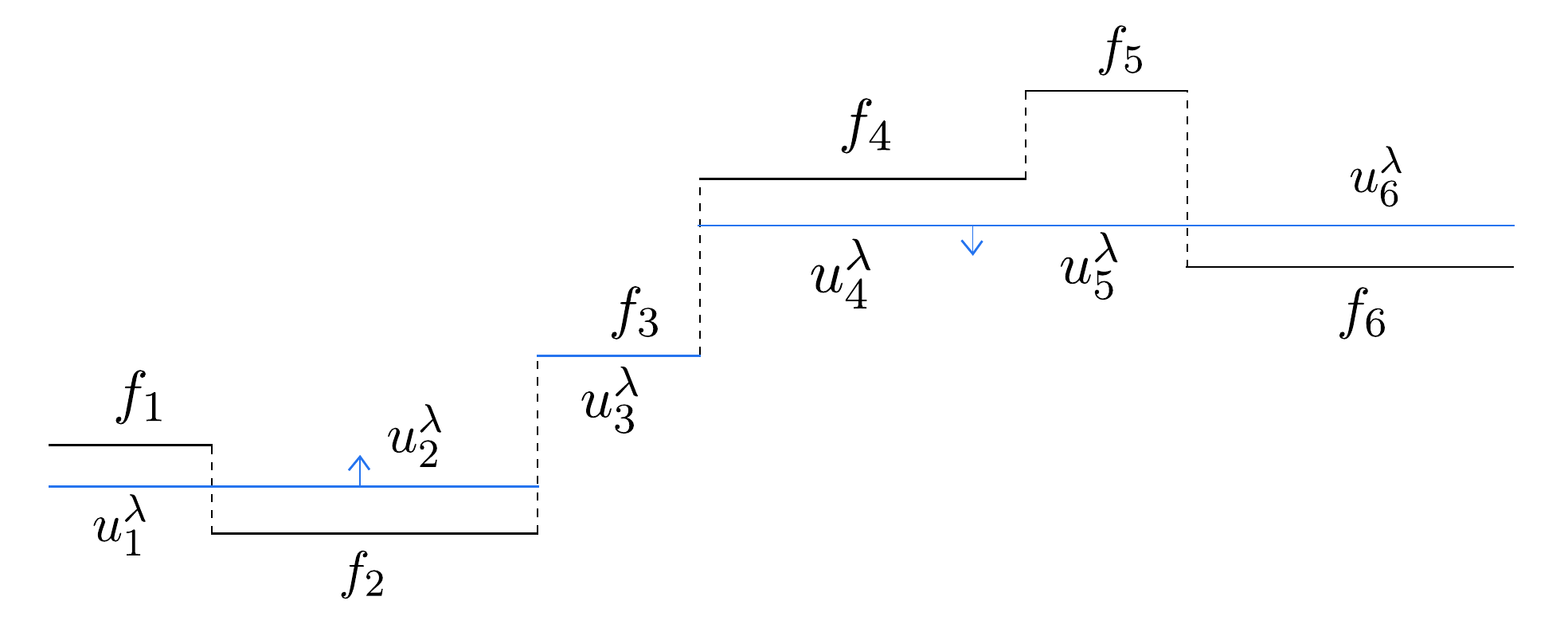}
\caption{The behavior of the solution for $\lambda\in(,\frac{1}{10},\frac{9}{14}]$ as $\lambda$ decreases.}
\end{figure}

The solution is now
\[
\begin{array}{lllll}
u_1^\lambda:=\frac{4}{3}+\frac{1}{6\lambda}\,,&&
u_2^\lambda:=\frac{4}{3}+\frac{1}{6\lambda}\,,&&
u_3^\lambda:=3\,,\\
&&&&\\
u_4^\lambda:=\frac{16}{3}-\frac{1}{6\lambda}\,,&&
u_5^\lambda:=\frac{16}{3}-\frac{1}{6\lambda}\,,&&
u_6^\lambda:=\frac{16}{3}-\frac{1}{6\lambda}\,.
\end{array}
\]

Notice that for $\lambda=\frac{1}{4}$ we have $u_1^\lambda=f_1$.
Thus, for $\lambda<\frac{1}{4}$, we have to consider the functional
\begin{align*}
G(v_1,v_2,v_3)&:= v_3-v_1 + \lambda[\, (v_1-2)^2+
    2(v_1-1)^2+|v_2-3|^2\\
&\hspace{0.5cm}+2(5-v_3)^2+(6-v_3)^2+2(v_3-4)^2 \,]\,.
\end{align*}
Hence, the solution remains equal to the previous ones.
For $\lambda=\frac{1}{10}$ we get $u_2^\lambda=u_3^\lambda$.
Then we consider the functional
\begin{align*}
G(v_1,v_2)&:= v_2-v_1 + \lambda[\, (2-v_1)^2+
    2(v_1-1)^2+|v_2-3|^2\\
&\hspace{0.5cm}+2(5-v_2)^2+(6-v_2)^2+2(v_2-4)^2 \,]\,.
\end{align*}

\begin{figure}[H]
\includegraphics[scale=0.6]{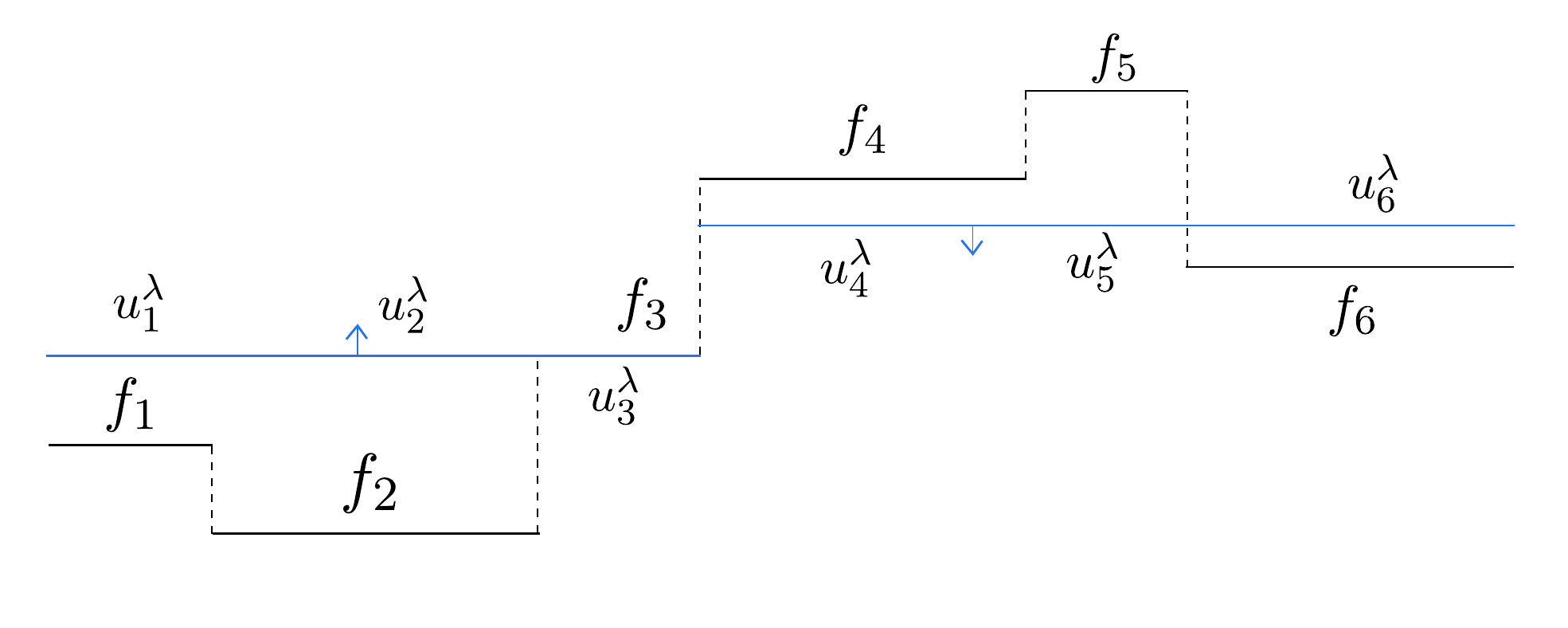}
\caption{The behavior of the solution for $\lambda\in(\frac{9}{122},\frac{1}{10}]$ as $\lambda$ decreases.}
\end{figure}

Such a functional is minimized by
\[
\begin{array}{lllll}
u_1^\lambda:=\frac{7}{4}+\frac{1}{8\lambda}\,,&&
u_2^\lambda:=\frac{7}{4}+\frac{1}{8\lambda}\,,&&
u_3^\lambda:=\frac{7}{4}+\frac{1}{8\lambda}\,,\\
&&&&\\
u_4^\lambda:=\frac{24}{5}-\frac{1}{10\lambda}\,,&&
u_5^\lambda:=\frac{24}{5}-\frac{1}{10\lambda}\,,&&
u_6^\lambda:=\frac{24}{5}-\frac{1}{10\lambda}\,.
\end{array}
\]

Finally, for $\lambda\leq\frac{9}{122}$ we have that the solution is given by
\[
u_1^\lambda=u_2^\lambda=u_3^\lambda=u_4^\lambda=u_5^\lambda=u_6^\lambda:=\frac{31}{9}\,.
\]

\begin{figure}[H]
\includegraphics[scale=0.6]{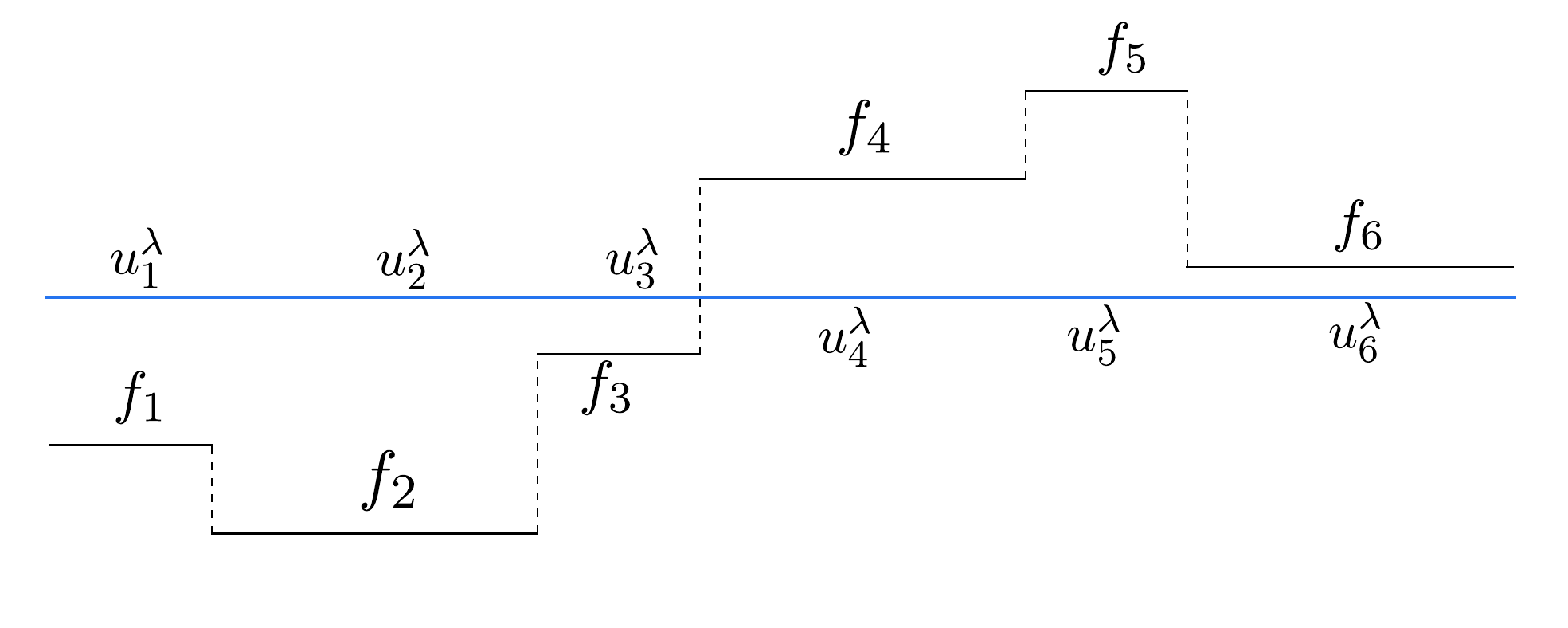}
\caption{The behavior of the solution for $\lambda<\frac{9}{122}$.}
\end{figure}

\begin{remark}
The previous example allows us to identify some properties of the solution $u^\lambda$:
\begin{enumerate}
\item it is \emph{not} true that if $u_i^{\bar{\lambda}}=f_i$, then 
$u_i^\lambda=f_i$ for all $\lambda\geq\bar{\lambda}$,
\item the function $\lambda\mapsto u_i^\lambda$ is \emph{not} monotone in general.
Nevertheless, a change in the monotonicity can happen only if $\lambda=\lambda_i$ or $\lambda=\lambda_{i-1}$,
\end{enumerate}
\end{remark}

\begin{remark}
Let us denote by $u^{\lambda,p}$ the solution of problem \eqref{eq:ming} corresponding to $p$ and $\lambda$.
Although we know that, for every $\lambda$ fixed, $u^{\lambda,p}\rightarrow v$ as $p\searrow1$, where $v$ is a solution of the problem \eqref{eq:ming} corresponding to $\lambda$ and $p=1$, we cannot apply directly our method to find $v$, since
analytic computations are difficult to perform in the case $p\in(1,2)$.
Nevertheless, a finer analysis of the behavior of the solution $u^{\lambda,p}$ for $p\in(1,2)$ is currently under investigation.
\end{remark}


\bigskip
\bigskip
\noindent
{\bf Acknowledgments.}
The author wishes to thank Irene Fonseca for having introduced him to the study of this problem and for helpful discussions during the preparation of the paper.

The author warmly thanks the Center for Nonlinear Analysis at Carnegie Mellon University for its support
during the preparation of the manuscript.
The research was funded by National Science Foundation under Grant No. DMS-1411646.


\bibliographystyle{siam} 
\bibliography{bibliografia}

\end{document}